\documentclass{article}
\usepackage{amsmath, amssymb, amsthm, bbm}
\usepackage{enumitem}
\usepackage[T1]{fontenc}
\usepackage{tikz}
\usepackage[UKenglish]{isodate}
\usepackage[font=small,labelfont=it]{caption}
\usepackage{subcaption}
\usepackage{hyperref}
\usepackage{multirow}
\usepackage{fancyvrb}

\usepackage{natbib}
\usepackage{enumitem}
\usepackage{tikz}
\usepackage{graphicx}
\usepackage{xcolor}
\usepackage[UKenglish]{isodate}
\usepackage{graphicx}
\usepackage[font=small,labelfont=it]{caption}
\usepackage{subcaption}
\usepackage{hyperref}
\usepackage{enumitem}
\usepackage{multirow}
\usepackage{fancyvrb}

\renewcommand{\phi}{\varphi}
\newcommand{\eps}{\varepsilon}

\newcommand{\om}{\omega}
\newcommand{\Om}{\Omega}
\newcommand{\lam}{\lambda}
\newcommand{\Ll}{\left}
\newcommand{\Rr}{\right}
\newcommand*\diff{\mathop{}\!\mathrm{d}}

\newcommand{\paren}[1]{\left( \left. #1 \right. \right)}

\newcommand{\rhoseason}{m_{\rho}}
\newcommand{\rhoseasondiscrete}[1]{m_{\rho,#1}}
\newcommand{\croch}[1]{\left[ \left. #1 \right. \right]}

\newcommand{\N}{\mathbb{N}}
\renewcommand{\P}{\mathbb{P}}

\newcommand{\R}{\mathbb{R}}

\newcommand{\Ff}{\mathcal{F}}
\newcommand{\Hh}{\mathcal{H}}

\newcommand{\Kk}{\mathcal{K}}

\newcommand{\Nn}{\mathcal{N}}
\newcommand{\Pp}{\mathcal{P}}

\newcommand{\Xx}{\mathcal{X}}

\newcommand{\rhowinter}{\rho}

\newtheorem*{theorem*}{Theorem}
\newtheorem{theorem}{Theorem}[section]
\newtheorem{atheorem}{Theorem}
\newtheorem{proposition}[theorem]{Proposition}
\newtheorem{corollary}[theorem]{Corollary}
\newtheorem{remark}[theorem]{Remark}
\newtheorem{lemma}[theorem]{Lemma}

\newtheorem{definition}[theorem]{Definition}

\title{The Yoccoz-Birkeland livestock population model
coupled with random price dynamics}

\author{Riccardo~Ceccon\thanks{Scuola Normale Superiore, Pisa, Italy. Email address: riccardo.ceccon@sns.it} 
\and Giulia~Livieri\thanks{Scuola Normale Superiore, Pisa, Italy. Email address: giulia.livieri@sns.it} \and Stefano~Marmi\thanks{Scuola Normale Superiore, Pisa, Italy. Email address: stefano.marmi@sns.it}}

\date{\today}
\begin{document}
\maketitle
\begin{abstract}
We study a random version of the population-market model proposed by Arlot, Marmi and Papini in \cite{Arl19}. The latter model is based on the Yoccoz–Birkeland integral equation and describes a time evolution of livestock commodities prices which exhibits endogenous deterministic stochastic behaviour. We introduce a stochastic component inspired from the Black-Scholes market model into the price equation and we prove the existence of a random attractor and of a random invariant measure.  We  compute numerically the fractal dimension and the entropy of the random attractor and we show its convergence to the deterministic one as the volatility in the market equation tends to zero. We also investigate in detail the dependence of the attractor on the choice of the time-discretization parameter. We implement several statistical distances to quantify the similarity between the attractors of the discretized systems and the original one. In particular, following a work by Cuturi \cite{CU13},  we use the Sinkhorn distance. This is a discrete and penalized version of the Optimal Transport Distance between two measures, given a transport cost matrix.
\end{abstract}
{\small
}

\section{Introduction}
\label{sec:Introduction}
One of the most outstanding phenomena in ecology is given by the statistically cyclical variations of small Arctic rodents (see, e.g., \cite{Han93, And21}). The amplitude of these cycles varies widely and seemingly chaotically.	In 1998, Yoccoz and Birkeland (\cite{YB98}) proposed a model for the evolution of the population of \emph{Microtus epiroticus} (the sibling vole) on the Svalbard Islands in the Arctic Ocean. This species presents a high fertility rate that has a strong dependence on seasonal factors and on the population density. The peculiarity of this model is that it exhibits an \emph{endogenous} chaotic behaviour: it tries to explain the strong annual oscillations presented by the Microtus epiroticus population only by exploiting the biological characteristics of the species and its interaction with the environment.  Remarkably, these oscillations are determined neither from the lack of food nor from the presence of significant predators, the main one being seagulls.\\
\indent In \cite{Arl04}, the authors studied the Yoccoz-Birkeland model via a mathematical analysis and some simulation experiments. Their main are the following: the Yoccoz-Birkeland model is able to reproduce a complex dynamics with a high sensitivity to initial conditions, only by the combination of density-dependent fertility, the lag due to the maturation age and a periodic seasonality. \cite{NPV12}, instead, proved the existence of periodic points for the discrete version of the model and performed numerical simulations with a special emphasis to some values for the model parameters.\\
\indent 	Persistent approximately periodic fluctuations are also a feature of time-series of livestock commodities prices (\cite{Ros94}). Cobweb models show that non-linearities (\cite{CHI88, Hom94}) and simple expectation rules (\cite{Hom13}) may lead to chaotic deterministic price fluctuations. In \cite{Arl19}, the authors coupled the Yoccoz-Birkeland model with a \emph{deterministic} equation modelling the price dynamics of a livestock commodity market. The main idea of this model is the following. A cattle population is split at the birth into reproducing females and cattle for butchery. The splitting ``strategy'' is determined by very simple naive expectations of the breeder, namely by the spot price of the meat, whereas the logarithmic derivative of the price is driven by the unbalance between the demand and supply of the meat. As in \cite{Arl04}, the model for the population dynamics accounts explicitly for seasons (or artificial synchronization of births) and maturation lags.\\
\noindent The major outcomes of their study are the following. First, the model gives rise to a chaotic time evolution of price by simply connecting the percentage of reproducing females with the price equation. Second, they show global existence and uniqueness for initial value problems together with some useful estimates on the solutions. Third, they give simple sufficient conditions that ensure the existence of a global attractor containing at least a non-trivial periodic solution. Finally, they perform some numerical experiments in which they show that the attractor has sensitive dependence on initial conditions and non-integer dimension, and that the influence of the market on the model is a crucial factor for the dynamics.\\
\indent	In this paper, we extend the model of \cite{Arl19}: in particular, the price dynamics of the livestock commodity market is now described by a \emph{Stochastic Differential Equation (SDE, henceforth)}, in a Black-Scholes-like fashion (\cite{B900} and \cite{BS73}), instead of a deterministic one. The aim is that of proving also in the new framework the asymptotic results produced in \cite{Arl19}, in particular the existence of a global attractor. Thanks to the bounded dependence of the population function with respect to the price, we manage to prove that the population size is bounded and Lipschitz. Nonetheless, the presence of the stochastic component on the price dynamics compromises the usage of the theorem employed in \cite{Arl19} to prove the existence of a bounded attractor. Here, instead, we use the concepts of \emph{Random Dynamical System (RDS, henceforth)} and \emph{Random Attractor}, which are based on the so-called \emph{pull-back} approach (see, e.g., \cite{CR02} and \cite{CRFL92}), to prove the existence of a random attractor on the population component of the phase space. In addition, to prove a more powerful asymptotic result on the population-price dynamical system, we use the concept of \emph{Random (Invariant) Measure} and a version of the Prohorov Theorem for Random Measures (\cite{CR02}) to prove the existence of an invariant random measure. Moreover, we complement the numerical experiments in \cite{Arl19}: the aim is that of simplifying, i.e., reducing as far as possible the dimension of the phase space of the dynamical system obtained by discretizing the model proposed by \cite{Arl19} with the intention of preserving as much as possible the chaotic behaviour observed in the previous study. Toward this end, we allow also for a non-integer number of integration steps. We show graphically the effects of reducing the number of integration steps per year, which appears to be smoothing and simplifying the attractor of the dynamical system. Due to this simplification, we introduce a new set of parameters, such that the plot for a smaller number of integration steps per year changes the least possible. In addition, we employ the concept of Optimal Transport Distance and of Sinkhorn Distance, referring to \cite{CU13}. Thanks to these distances, we are able to formalize and compute explicitly the geometrical difference between the attractors. A plot of this distance, together with the entropy and the fractal dimension, as a function of the number of integration steps per year, is then shown.\\
\indent The paper is organized in the following way. In Section \ref{sec:yoccoz_model}, we recall the Yoccoz-Birkeland model coupled with deterministic price dynamics developed in \cite{Arl19}. Section \ref{sec:model_Yoccoz_Birkeland} describes the random model obtained by adding a diffusion term to the price dynamics and some preliminary estimates on the time evolution of the system are derived in Section \ref{sec:first_analysys}. The existence of a global random attractor is established in Section \ref{sec:existence}. Section \ref{sec:numerical} is devoted to the numerical study of the 
discretizaton of the deterministic model and the sensitivity of the attractor on the choice of 
integration step, whereas Section \ref{sec:raandom_num} contains the numerical study of the random attractor. The two appendices contain the fundamental notions on RDS needed and the application to the logistic map of the numerical methodologies used in Appendix \ref{app:logistic_attractors} to quantify the distance between attractors.  
\section{The Yoccoz-Birkeland model and the coupling with price dynamics}\label{sec:yoccoz_model}
This section summarises the model introduced in \cite{Arl19} for the time evolution of livestock commodities prices; it hinges on the Yoccoz-Birkeland integral equation \cite{YB98} (see Equation~\eqref{eq:model_1} below) and on the subsequent analysis in \cite{Arl04, NPV12}.\\
A cattle population is divided into two parts: the first one comprises females for reproduction whereas the second one the cattle for butchery (i.e., all the males plus some of the females). Then, the mechanism governing the time evolution of the livestock population and of the meat price is the following: 
\begin{enumerate}
\item[(1)] At the birth of some babies, part of the newborn females are put in the reproduction line, and the remaining newborn females are put in the butchery line together with all newborn males. The fraction $R$ of newborn females that will reproduce is chosen by the breeder, and is only determined by the price of meat at birth time; the breeder can choose either a short-term strategy or a long-term strategy. Below the dependence upon the price $P$ is denoted by $R(P)$.
\item[(2)] In the reproducing line, females between ages $A_0$ and $A_1$ have children. Their fertility can be affected by seasons, or because births are synchronized by the breeder (through a function $m_{\rho}(t)$). Reproducing females older than $A_1$ (hence, non fertile) are not taken into account anywhere in the model - in particular, they are not butchered.
\item[(3)] In the butchery line, cattle can be butchered between ages $\Omega_0$ and $\Omega_1$. So, only the (alive) butchery population between ages $\Omega_0$ and $\Omega_1$ can count as a ``supply'' for the market. 
\item[(4)] The price evolution is a simple function of the supply (which comes from the butchery line population dynamics) and the demand (which depends only on the price).
\end{enumerate}

\noindent Before describing the model, for the reader's convenience, we sum up in the next subsection the notation and the terminology which will be used throughout this paper. For the sake of clarity and possibility of comparison, we use the notation as in \cite{Arl19}. 
\subsection{Terminology and notations}\label{subsec:terminology_and_notation}
\begin{enumerate}[label=(N\arabic*)]
\item\label{itm:N1} $t$ is the time measured in years.
\item\label{itm:N2} $N_r(t)$ is the total population of \emph{mature females} that are in the \emph{r}eproducing line and can give birth to pups at time $t$.
\item\label{itm:N3} $N_b(t)$ is the total population of cattle that is \emph{suitable for butchery} at time $t$ (both males and non-reproducing females, old enough and in the butchery line).
\item\label{itm:N4} $R(P)$ is the fraction of \emph{newborn females} that are put in the reproducing line when the price of meat is $P$ when they are born.
\item\label{itm:N5} $A_0$ is the age from which \emph{females} can have children (i.e., the age of sexual maturity plus the length of the first gestation).
\item\label{itm:N6} $A_1$ is the maximal age at which \emph{females} can give birth to children (i.e., the age of sexual infertility plus the length of the last gestation).
\item\label{itm:N7} $\Omega_0$ is the minimal age at which the cattle (\emph{male or female}) can be butchered.
\item\label{itm:N8} $\Omega_1$ is the maximal age at which the cattle (\emph{male or female}) can be butchered. 
\item\label{itm:N9} $m(N)$ is the average annual \emph{female} (resp.~\emph{male}) fertility of each mature female when the total population is $N$, i.e., the average number of female (resp. male) babies per year for a single mature female. Typically it is a decreasing function; see Equation~(2.3) in \cite{Arl19} for a concrete example. We assume a sex ratio equal to $1/2$, i.e. the average number of male babies is equal to the average number of female babies. Hence, $m(N)$ is half of the average annual fertility.
\item\label{itm:N10} $m_{\rho}(t)$ is the 1-periodic step function (with $\int_0^1 m_{\rho}(t)\,dt=1$) that accounts for a possible modulation of fertility during each year (e.g., births synchronization or seasonal effects).
\item\label{itm:N11} $P(t)$ is the market price of meat at time $t$.
\item\label{itm:N12} $D(P)$ is the demand of the market (per time unit) when the price of meat is $P$ (typically a decreasing function of $P$).
\item\label{itm:N13} $S(t)$ is the supply to the market (per time unit) at time $t$ (typically proportional to $N_b(t)$).
\item\label{itm:N14} $\lambda$ is a ``temperature'' parameter of the meat market: 
higher values of $\lambda$ correspond to bigger price variations
in response to the same demand/supply imbalance.
\item\label{itm:N15} $F(D,S)$ is the function of demand and supply that rules the meat price dynamics.
\item\label{itm:N16} $B_f(t)$ is the density of newborn female cattle at time $t$ (i.e., $B_f(t)dt$ females are born between $t$ and $t+ dt$).
\item\label{itm:N17} $B_m(t)$ is the density of newborn male cattle at time $t$.
\item\label{itm:N18} $B_r(t)$ is the density of newborn (female) cattle that are put in the reproducing line at time $t$.
\item\label{itm:N19} $B_b(t)$ is the density of newborn cattle that are put in the butchery line at time $t$.
\end{enumerate}
\subsection{Deterministic population and price dynamics}\label{subsec:population_and_price_dynamics}
The above assumptions on the coupling between the cattle population evolution and the deterministic market model made in \cite{Arl19} (see \cite{Arl19}, Subsections~3.2, 3.3 for a detailed derivation) give rise to the following set of equations for the time evolution of the reproducing females population $N_r$, of the price $P$ and of the meat supply $S$:
\begin{align}
N_r(t)                  &=\int_{A_0}^{A_1}m_{\rho}(t-a) m(N_r(t-a)) N_r(t-a) R(P(t-a))\,da\label{eq:model_1}\\
\frac{P^{'}(t)}{P(t)}   &=\lambda F(D(P(t)), S(t))\label{eq:model_2}\\
S(t)                    &=\frac{1}{\Delta\Omega} \int_{\Omega_0}^{\Omega_1} m_{\rho}(t-a) m(N_r(t-a)) N_r(t-a)\left[2-R(P(t-a))\right]\,da\label{eq:model_3}
\end{align}
with $F(D, S) = \frac{(D-S)}{(D+S)}$, $m : [0,+\infty)\rightarrow [0,+\infty)$, $m_{\rho} : \mathbbm{R} \rightarrow [0,+\infty)$, $R:[0,+\infty)\rightarrow [0,1]$, $D:[0,+\infty)\rightarrow [0,+\infty)$, $A_1 > A_0 > 0$, $\Omega_1 > \Omega_0 > 0$, $\Delta \Omega = \Omega_1 - \Omega_0$ and $\lambda>0$.\\
\indent Notice that the integral evolution Equation~\eqref{eq:model_1}, which is a standalone equation in the Yoccoz-Birkeland model,  is coupled with the \emph{differential} equation in Equation~\eqref{eq:model_2} describing the price of a livestock commodity; the latter is driven by the unbalance between its demand and supply.\\
\indent Under reasonable assumptions on $m_{\rho}$, $m$, $R$ and $D$ (see \cite{Arl19}, beginning of Section~4, and Equations~(4.7)--(4.9)), the authors show that a unique solution to Equations~\eqref{eq:model_1}--\eqref{eq:model_3} exists, is globally defined and satisfies some (uniform with respect to the initial condition) estimates: $N_r$ and $S$ are globally bounded and the component $N_r$ is Lipschitz continuous on $[t_0,+\infty)$ for some $t_0$. Moreover, they prove the existence of a global attractor and of a non-trivial periodic solution (see \cite{Arl19}, Theorem A and Theorem B, respectively). Finally, via a numerical investigation, they show that the global attractor is indeed a strange attractor: it has a fractal dimension of $\approx 1.53$ and the Kolmogorov-Sinai entropy is positive. Moreover, they showed that the attractor is persistent but its chaotic behaviour depends also on the time evolution of the price in an essential way, a feature which was completely absent in the original Yoccoz-Birkeland model. In particular, they show that if the price dynamic is constant, then the resulting orbit becomes quasi-periodic. Figure~\ref{fig:yoccoz_arlot} displays an example of attractor for the model of \cite{Arl19}.
\begin{figure}[h]
\centering
\includegraphics[width=0.8\linewidth]{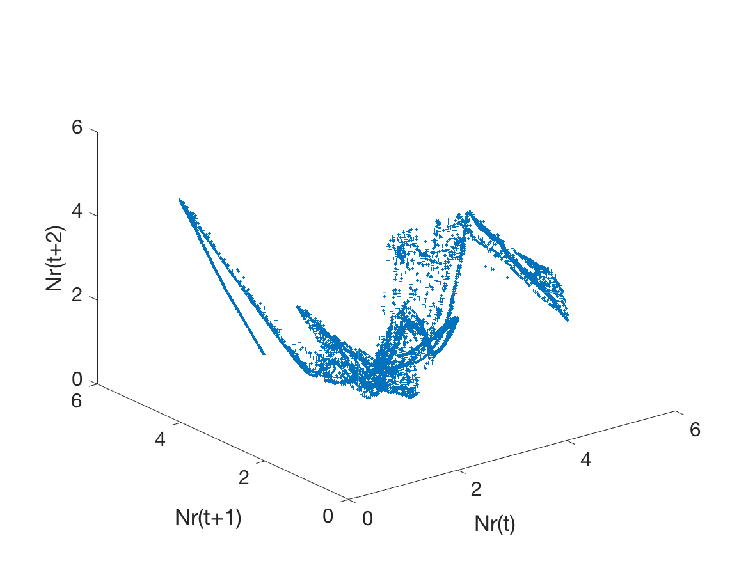}
\caption{Three dimensional plot of $(N_r(t), N_r(t+1), N_r(t+2))$ with $\rho=0.30,\,\gamma=8.25,\,A_0 = 0.18$ and $m_0 = 50$.}
\label{fig:yoccoz_arlot}
\end{figure}

\section{The Yoccoz-Birkeland model and the coupling with random price dynamics}\label{sec:model_Yoccoz_Birkeland}
The model we propose is obtained from Equations~\eqref{eq:model_1}--\eqref{eq:model_3} by adding a stochastic Brownian component with a constant volatility $\sigma>0$ to price dynamics (Equation~\eqref{eq:model_2}). The constancy of the volatility is consistent with two popular models used in the financial literature, i.e., the Bachelier and the Black and Scholes (see \cite{B900} and \cite{BS73}, respectively). 
Precisely, the price is described by the following \emph{Stochastic Differential Equation} (SDE):
\begin{equation}
     dP(t) = P(t) \lambda F(D(P(t)), S(t))\,dt + P(t) \sigma\,dW_t\label{eq:model_2_prime},
\end{equation}
or, equivalently, in its logarithmic form: 
\begin{equation}
    dQ(t) = \left(F(D(P(t)), S(t)) - \frac{\sigma^2}2\right)\,dt+ \sigma\,dW_t\label{eq:model_2_second}.\\
\end{equation}
Therefore, the model we are going to study is the following:
\begin{align}
N_r(t)                  &=\int_{A_0}^{A_1}m_{\rho}(t-a) m(N_r(t-a)) N_r(t-a) R(P(t-a))\,da\label{eq:model_1_prime}\\
dP(t)                   &= P(t) \lambda F(D(P(t)), S(t))\,dt + \sigma P(t)\,dW_t\label{eq:model_2_prime}\\
S(t)                    &=\frac{1}{\Delta\Omega} \int_{\Omega_0}^{\Omega_1} m_{\rho}(t-a) m(N_r(t-a)) N_r(t-a)\left[2-R(P(t-a))\right]\,da,\label{eq:model_3_prime}
\end{align}
where $(W_t)_{t \geq 0}$ is a standard Brownian motion.\\

\noindent In order to analyse the effect of the noise on the attractor of a deterministic dynamical system, we will use the \emph{pull-back approach} (see, e.g., \cite{CHEK11}). To this end, we introduce the two-sided Wiener measure before proceeding with the mathematical analysis of the model. Precisely, on some probability space, we take two independent copies of the Brownian motion $(W_t)_{t \geq 0}$, say $(W_t^{(i)})_{t \geq 0}$ with $i = 1, 2$, and we define the two-sided Brownian motion:
\begin{equation*}
    W_t = W_t^{(1)}\,\,\text{for}\,\,t \geq 0,\,\,W_t = W_{-t}^{(2)}\,\,\text{for}\,\,t \leq 0.
\end{equation*}
We call $\mathcal{P}$ its law on Borel sets of the space of real-valued continuous functions on $\mathbb{R}$ that are null at zero, i.e. the space $\mathcal{C}_0(\mathbb{R};\mathbb{R})$. This is the two-sided Wiener measure. Now, consider the canonical space $\Omega = \mathcal{C}_0(\mathbb{R};\mathbb{R})$ with Borel $\sigma$-field and two-sided Wiener measure $\mathcal{P}$. On $(\Omega, \mathcal{F}, \mathcal{P})$ we consider the canonical two-sided Brownian motion defined as  
\begin{equation*}
    W_t(\omega) = \omega(t)\quad\omega \in \mathcal{C}_0(\mathbb{R};\mathbb{R})
\end{equation*}
and we interpret the stochastic Equation \eqref{eq:model_2_prime} on the canonical space.\\
\indent In the next section, we prove some rigorous preliminary results on the Yoccoz-Birkeland model coupled with random price dynamics described by Equations~\eqref{eq:model_1_prime}--\eqref{eq:model_3_prime}. We will closely follow Section 4 of \cite{Arl19}, but at the same time account for the stochastic nature of the price dynamics.

\section{A first analysis of the model}\label{sec:first_analysys}
\noindent Following \cite{Arl19}, we impose the following assumptions on the seasonality $m_{\rho}$, the fertility $m$,  the fraction $R$ of newborn females put in the reproducing line and the demand function $D$:
\begin{enumerate}[label=(A\arabic*)]
\item\label{itm:A1} $m_{\rho}:\mathbb{R}\rightarrow\mathbb{R}$ is a non-negative, bounded, $1$-periodic function such that $\int_0^1m_{\rho}(t)\, dt = 1$ and we let $m_{\rho}(t)\le \mu_{\text{max}}$ and
\begin{equation*}
0<c_0\le\int_{A_0}^{A_1}m_{\rho}(t-a)\, da\le c_1 \qquad \forall t.
\end{equation*}
\item\label{itm:A2} $m:\left[0,+\infty\right)\to\mathbb{R}$ is a continuous function that satisfies
\begin{equation*}
\frac{m_0}{2} \min\{1, N^{-\gamma}\} \le m(N) \le m_0 \min\{1, N^{-\gamma}\} \qquad \forall N > 0
\end{equation*}
with $m_0>0$ and $\gamma\ge 1$.
\item\label{itm:A3} $R:\left[0,+\infty\right)\rightarrow\mathbb{R}$ is a continuous function such that $R_0\le R(P)\le R_1$ for all $P \geq 0$ and some constants $R_1,R_0>0$ with $R_1 \leq 1$.
\item\label{itm:A4} $D:\left[0,+\infty\right)\rightarrow\mathbb{R}$ is a strictly decreasing and locally Lipschitz continuous function such that $D(+\infty)=0$, and we set $D_{0}=D(0)$.
\end{enumerate}
\noindent We can now define the phase-space of the dynamical system and the notion of solution of Equations~\eqref{eq:model_1_prime}--\eqref{eq:model_3_prime}.\\
\indent We let $T_0 := \min\{A_0, \Omega_0\}$, $T_1 := \max\{A_1, \Omega_1\}$ and we set:
\begin{equation}\label{eq:defition_of_X}
    \mathcal{X}:=\mathcal{X}_1 \times \mathcal{X}_2:=C_{\text{int}}^{0}([-T_1,0];[0,+\infty)) \times C^{0}([-T_1,0];[0,+\infty)),
\end{equation}
where $C_{\text{int}}^{0}([-T_1,0];[0,+\infty))$ is the space of the positive-valued continuous functions on $[-T_1, 0]$ such that:
\begin{equation}\label{eq:condition_on_N_0}
    N_r(0) = \int_{A_0}^{A_1}   m_{\rho}(-a) m(N_r(-a)) N_r(-a) R(P(-a))\,da.
\end{equation}
The space $\mathcal{X}$ is a complete metric space with respect to the distance induced by the norm:
\begin{equation*}
    \|(N, P)\|_{\mathcal{X}}:=\|N\|_{\infty} + \|P\|_{\infty}:=\text{esssup}_{s \in [-T_1,0]} |N(s)| + \sup_{s \in [-T_1,0]} |P(s)|.
\end{equation*}

\begin{definition}
Let $(N_r^0, P^0) \in \mathcal{X}$, and $t_0, T \in \mathbb{R}$ with $t_{0}<T $. A solution of Equations \eqref{eq:model_1_prime}--\eqref{eq:model_3_prime} with initial data $(N_r^0, P^0)$ is a couple $(N_r,P): \Omega\times\left[t_0 - T_1, T\right) \rightarrow \mathbb{R}^2$ such that $N_r(\omega)|_{\left[t_0,T\right)}$ is continuous, $P(\omega)|_{\left[t_0,T\right)}$ is continuous, $N_r$ and $P$ satisfy Equations \eqref{eq:model_1_prime}--\eqref{eq:model_3_prime} for $t \in \left[t_0, T\right)$, while $N_r(\omega,t_0+a)=N_r^0(\omega,a)$ and $P(\omega,t_0+a)=P^0(\omega,a)$ for $a\in\left[-T_1,0\right)$.
\end{definition}

\noindent The assumptions \ref{itm:A1}-\ref{itm:A4} guarantee that the model has a unique globally defined solution:

\begin{proposition}[cfr.~\cite{Arl19}, Proposition 4.2]\label{prop:bounded_solutions}
Let $(N_r^0, P^0) \in \mathcal{X}$ and $t_0 \in \mathbb{R}$ be given.
Then there exists a unique solution pair $(N_r,P):\Omega\times\left[-T_1+t_0,+\infty\right)\to\mathbb{R}^2$ of Equations \eqref{eq:model_1_prime}--\eqref{eq:model_3_prime} with initial data $(N_r^0, P^0)$.
Moreover, $N_r, P$ are non-negative and $\mathcal{P}$-a.s. it holds that:
\[
\begin{aligned}
&0\le N_r(t) \le N_{\max}           \quad \forall t\ge t_0 \\
& \bigl\lvert N_r(t)-N_r(s) \bigr\rvert \le L_1 |t-s| \quad \forall t,s\ge t_0 \\
& 0 \le S(t) \le S_{\max}       \quad \forall t\ge t_0,
\end{aligned}
\]
where:
\begin{equation}\label{eq:constants_Nmax_L1S_max}
\begin{aligned}
& N_{\max} := m_0 R_1 c_1, \qquad L_1 := 2 m_0 R_1 \mu_{\text{max}} \\
& S_{\max} := m_0 \frac{2-R_0}{\Delta\Omega} \sup_{s\in[0,1]} \int_{\Omega_0}^{\Omega_1} m_{\rho}(s-a)\diff a.
\end{aligned}
\end{equation}
\end{proposition}
\begin{proof}
It's a simple adaptation of the proof of the analogous statement given in \cite{Arl19}, the main 
difference being that we have to determine $P$ on $[t_0,t_0+T_0)$ once $N_r$ and $S$ are extended on $\left[t_0-T_1,t_0+T_0\right)$ in a unique and continuous way. This follows from the stochastic version of the Cauchy-Lipschitz (or Picard-Lindel\"of) theorem. 
\end{proof}

\begin{remark}\label{rem:discontinuity_at_zero}
Equation \eqref{eq:model_1_prime} prescribes the value $N_r(t_0)$ which may be different from $N_r^0(t_0^-)$. Therefore, the solution component $N_r$ may have a jump discontinuity at $t_0$ even if the initial condition $N_r^0$ is continuous; this despite the fact that the solution is going to be Lipschitz continuous on $[t_0, \infty)$. This justifies the additional condition we imposed on $N_r^0(0)$ (Equation~\eqref{eq:condition_on_N_0}) while defining the space $\mathcal{X}$. Should this condition not be imposed, we would have that the shifted solution has a discontinuity at one point. Notice also that if $N_r^0(t)>0$ for all $t\in[-T_1,0)$, then $N_r(t)$ is going to be positive for all $t\in[t_0,\infty).$
\end{remark}

\noindent The following proposition establishes lower bounds for $N_r$ and $S$ which, together with the estimates in Proposition \ref{prop:bounded_solutions}, will allow us to define for the deterministic model a compact subset $K$ of $\mathcal{X}$ to which all the solutions belong after a sufficiently big period of time. However, notice that differently from the upper bounds, the time of first entry for the lower bounds is not uniform in the initial data. This will be important when discussing the existence of a random attractor.

\begin{proposition}[see.~\cite{Arl19}, Proposition 4.4]\label{prop:uniform_persistence_for_N}
Let $(N_r, P)$ be the solution of Equations \eqref{eq:model_1_prime}--\eqref{eq:model_3_prime} with initial data $(N_r^0, P^0) \in \mathcal{X}$ at time $t_0$.
\begin{enumerate}
\item
If $N_r(t) \le N_{\max}$ for a.a. $t \in [\hat{t} - A_1, \hat{t}]$ for some $\hat{t}  \geq t_0$, then
\[
N_r(t) \ge \frac{m_0R_0c_0}{2} \min\left\{ \inf_{[\hat{t}-A_1,\hat{t}]} N_r, N_{\max}^{1-\gamma}\right\}
\quad \forall t \in \bigl[ \hat{t}, \hat{t}+A_0 \bigr] 
 ,
\]
In particular this inequality holds for all $t \ge t_0 + A_1$ by Proposition~\ref{prop:bounded_solutions}.
\item
If $m_0R_0c_0>2$ and $\inf_{[\hat{t}-A_1,\hat{t}]} N_r \ge N_{\max}^{1-\gamma}$ for some $\hat{t} \ge t_0 + A_1$, then
\[
N_r(t) \ge N_{\min} \quad \text{and} \quad S(t) \ge S_{\min} \qquad \forall t\ge \hat{t},
\]
where:
\begin{small}
\begin{equation}\label{eq:Nmin_Smin}
N_{\min} := \frac{m_0R_0c_0}{2} N_{\max}^{1-\gamma} \quad \text{and} \quad
S_{\min} := m_0 \frac{2-R_1}{2\Delta\Omega} N_{\max}^{1-\gamma} \inf_{s\in[0,1]} \int_{\Omega_0}^{\Omega_1}\rhoseason(s-a)\diff a.
\end{equation}
\end{small}
\item
If $m_0R_0c_0>2$ and $ N_{0}(a) >0 $ for almost all $ a \in [-A_{1},0] $, then there exists $t^*\ge t_0$ such that
$N_r(t) \ge N_{\min}$ and $S(t) \ge S_{\min}$ for all $t\ge t^*$.
\end{enumerate}
\end{proposition}

\noindent The fact that there exists a time $t^{*}$ after which $N_r(t)$ belongs to a compact set $K$ enabled \cite{Arl19} to show that a similar property holds also for the price $P(t)$; cfr.~Proposition 4.5 and Corollary 4.6 in that paper. However, because of the stochastic nature of our price dynamics, in our case this is no longer true. In particular, this does not allow us to use the same techniques employed in \cite{Arl19} for the study of the attractors of the model. For this reason, we need to introduce the concept of \emph{Random Dynamical System (RDS, henceforth)}, of \emph{random attractors} and of \emph{random invariant measure}; see Appendix \ref{app:RDS_RA_RM}. In addition, we make the following remark.
\begin{remark}
From here on, the theoretical exposition will be made using $t_0=0.$ It can be observed that this does not entail any loss in generality, since we can translate a general solution $(N_r(t),P(t))$ at time $t_0$ to a solution $(N_r'(t),P'(t))=(N_r(t+t_0),P(t+t_0))$ at time $0$, which follows the same model with the seasonal periodic function $m_\rho$ ``shifted'' by $t_0$. Consequently, the result of the next two sections are true for every $t_0\in\mathbb{R}.$
\end{remark}

\section{Existence of a global random attractor}\label{sec:existence}
In this section, we prove the existence of a random attractor and of random invariant measure; we refer to Appendix \ref{app:RDS_RA_RM} for a brief review of the relevant concepts. The traditional  approach employed to analyse the effect of the noise on the attractor of a deterministic dynamical system is that of finding stationary solutions of the Fokker-Planck equation. Indeed, numerically it is often easier to integrate the system forward in time; see \cite{LM94}. However, the perturbation of a deterministic system by noise is such that the support of the (numerically) obtained probability density function corresponds to a neighborhood of the deterministic attractor. Whence, it is possible to get statistical information only, without any link with the geometry of the attractor; see Figure \ref{fig:det_vs_stoc} for a graphical display of the this claim. We will thus employ the so-called RDS approach: instead of integrating the system forward in time, we will run it from a distant point $s$ in the past until the current time $t$ where the system will be frozen. This approach is also called \emph{pullback approach}. 
\begin{figure}[!h]
\includegraphics[scale=0.32]{fig1.png}
\includegraphics[scale=0.32]{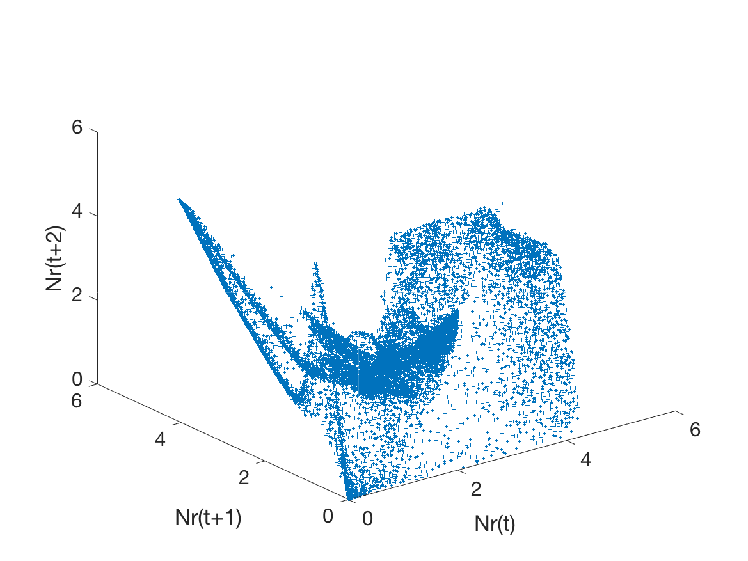}
\caption{Three dimensional plot of $(N_r(t), N_r(t+1), N_r(t+2))$ both for the deterministic (\emph{left panel}) and the stochastic (\emph{right panel}) case.}\label{fig:det_vs_stoc}
\end{figure}
\noindent \noindent First, we define the RDS object of our study (cfr. Definition~\ref{def:RDS}). The set of times is $T=\mathbb{R}_{+}$ and the phase-space $\mathcal{X}$ has been defined in Equation~\eqref{eq:defition_of_X}. The Borel $\sigma$-algebra of the phase-space is denoted by $\mathcal{B}$. The operator $\theta_t(\omega)(s) = \omega(t+s) - \omega(t)$ is the time shift of the Brownian motion, defined $\forall\,t\in\mathbb{R}$. Therefore, the operator $\varphi(t,\omega)$ acts on an element $(N_r, P) \in \mathcal{X}$ in the following way:
\begin{equation*}
\begin{split}
    \varphi(t,\omega)(N_r, P) = (N_r^{t}, P^{t})\,\,&\text{where}\,\,N_r^{t}(s) = N_r(t+s)\,\,\text{and}\,\,P^{t}(s) = P(t+s).
\end{split}
\end{equation*}
The dependence on the parameter $\omega$ of the variables on the \emph{r. h. s} is omitted. By abuse of notation, we will also write $\varphi(t,\omega) N_r = N_r^{t}$ and $\varphi(t,\omega) P = N_r^{t}$. We now state and prove the following lemma. 
\begin{lemma}
The RDS $\varphi(t, \omega):\mathcal{X}\rightarrow\mathcal{X}$ given by the stochastic coupled Birkeland-Yoccoz model in Equations~\eqref{eq:model_1_prime}-\eqref{eq:model_3_prime} is continuous.
\end{lemma}
\begin{proof}
We prove the continuity of the RDS by induction. On the interval $[0, T_0]$ the continuity of the $P$ component is a straightforward consequence of the continuous dependence on the initial condition for the SDE, whereas the continuity of the $N_r$ component is a consequence of the continuity of the integral operator. Inductively, we can suppose that if $(N_r^0,P^0)$ is arbitrarily close to $(\hat{N_r^0},\hat{P^0})$ in $\Xx,$ then $\phi(t,\om)(N_r^0,P^0)$ is arbitrarily close to $\phi(t,\om)(\hat{N_r^0},\hat{P^0})$ for $t\in[(k-1)T_0,kT_0]$.  The base case can be thus straightforwardly applied to $\phi(t,\om)(N_r^0,P^0)$ and $\phi(t,\om)(\hat{N_r^0},\hat{P^0})$ to obtain continuity also for $t\in[kT_0,(k+1)T_0]$.
\end{proof}

In order to find a global attractor for $\varphi$, one could be tempted to use Theorem~\ref{th:global_attractor} of Appendix \ref{app:RDS_RA_RM}. However, it cannot be applied directly to our RDS because the price function $P(t)$ is not $\omega$-wise bounded; instead, $N_r$ belongs to the interval $[0, N_{max}]\,\forall\,t \geq t_0$. For this reason, we aim at finding a random attractor $A(\omega)$ defined similarly to that in Definition~\ref{def:globally} but with the difference that $A(\omega) = D(\omega) \times \mathcal{X}_2$, where $D(\omega) \subset \mathcal{X}_1$ a compact set on the first component of the space $\mathcal{X}$. Precisely, we have the following definition:

\begin{definition}[Globally attracting set]
\label{def:globally_new}
Let $\varphi$ be a RDS such that there exists a random set $A(\omega):=D(\omega) \times \mathcal{X}_2$, for some random compact set $D(\omega)$, satisfying to the following conditions:
\begin{enumerate}
    \item $\varphi(t, \omega)A(\omega) = A(\theta_t\omega)\,\forall\,t>0$;
    \item $A$ attracts every bounded deterministic set $B \subset \mathcal{X}$.
\end{enumerate}
Then $A$ is said to be a universally or globally attracting set for $\varphi$ onto the first component od the phase-space $\mathcal{X}$.
\end{definition}
\noindent Notice that that asking that $A(\omega)$ attracts every bounded set $B \subset \mathcal{X}$ is equivalent to asking that $D(\omega)$ attracts every bounded set $B \subset \mathcal{X}_{1}$. We give the following
\begin{definition}
Given $B\subset \Xx_1,$ we define $\Omega^1_B(\om)$ as the projection onto $\Xx_1$ of the set $\Omega_{B\times\Xx_2}(\om)$.
\end{definition}

In particular, the following proposition holds true. 
\begin{proposition}
\label{prop:global_new}
Let $\phi$ be a RDS on the Polish space $\Xx$. Suppose there exists a compact set $K_1(\om)$ such that $K_1\times\Xx_2$ absorbs the product between every bounded non-random set $B \subset \Xx_1$ and $\Xx_2$. Then the set
$$A(\om) = \overline{\bigcup_{B\subset\Xx_1} \Omega^1_B(\om)}\times \Xx_2$$
is a global attractor for $\phi$ in the sense of the Definition~\ref{def:globally_new}.
\end{proposition}
\begin{proof}
See the proof of Theorem 3.11 in \cite{CRFL92}. 
\end{proof}

\noindent Thus, we can now define the set 
\begin{equation}\label{eq:K_1}
\begin{split}
     K_1:=\big\{N_r^0\in\Xx_1: 0\le N_r^0(t)\le N_{\max}\ &\text{for all} \ t\in[-T_1,0], \ \text{and}\\ 
               &N_r^0 \ \text{is} \ 2m_0R_1\mu_{\max}\text{-Lipschitz-continuous} \big\};
\end{split}
\end{equation}
see Proposition~\ref{prop:bounded_solutions}. In particular, we have that:
\begin{proposition}\label{prop:existence_globally}
The set $K_1$ is a compact set and it satisfies the assumptions of Proposition~\ref{prop:global_new}. As a consequence, there exists a global attractor $A(\omega)$ onto the first component of the space $\mathcal{X}$. 
\end{proposition}
\begin{proof}
The set $K_1$ is compact thanks to Ascoli-Arzelà's Theorem. In addition, thanks to Proposition~\ref{prop:bounded_solutions} the set $K_1\times\Xx_2$ absorbs the product between every bounded nonrandom set $B \subset \Xx_1$ and $\Xx_2$. Indeed, for every $(N_r^0,P^0)\in\Xx$ and for all $\om\in\Om$ it holds that $\phi(t,\om)(N_r^0,P^0)$ belongs to $K_1\times\Xx_2$ for every $t\ge T_1.$
\end{proof}

\noindent Notice that the notion of random attractor is \emph{global} and there is no a definition of basin of attraction in the random framework (in the deterministic setup, the attractor's basin of attraction is given by the set of functions $N_r$ that are positive on the interval $[-T_1,0)$; see~\cite{Arl19}, Section 4).\\
\indent For this reason, we follow an averaging procedure described in detail in Subsection~\ref{subsec:invariant} on invariant measures on random sets. In particular, we want to apply Theorem~\ref{th:prohorov_random}. To this end, we have to define a suitable probability distribution $\nu$ for the initial values $(N_r^{0}, P^{0}) \in \mathcal{X}$. We define it component-wise as the conditional product of a measure $\nu_1(N_r^{0}|P^{0})$ on $\mathcal{X}_1$ and of a measure $\nu_2(P^{0})$ on $\mathcal{X}_2$. Notice that there is no a standard recipe to construct such measures: in this work, we define them via the Brownian motion's law. For $\nu_2$ we can choose any probability measure on the set of continuous functions: we pick the law of the absolute value of a two-sided Brownian motion starting at $t=-T_1$. In order to define $\nu_1$, we have to remember that for a fixed $P^{0} \in \mathcal{X}_2$ every function $N_r^{0} \in \mathcal{X}_1$ satisfies Equation~\eqref{eq:condition_on_N_0}, which does not allow for discontinuities at $t=0$. Thus, we introduce the stochastic process $(N_r^{0}(t))$ and we set it equal to the absolute value of a two-sided Brownian motion starting at $t=-T_1$ for $t \in [-T_1,-T_0]$ and to a linear interpolant between the point $N_{-T_{0}}$ and the point
\begin{equation*}\label{eq:second_point_affine}
    N_r^{0}(0) = \int_{A_0}^{A_1} m_{\rho}(-a) m(N_r^{0}(-a)) N_r^{0}(-a) R(P^{0}(-a))\,ds
\end{equation*}
on $[-T_0, 0]$. We define $\nu_1(\,\cdot\,| P^{0})$ as the law of such a process on $\mathcal{X}_1$ and we construct $\nu$ in the following way:
\begin{equation*}\label{eq:measure_nu}
    \nu(A) := \int_{\mathcal{X}_1}\int_{\mathcal{X}_2} I_{A}(N_r^{0}, P^{0})\,d\nu_1(N_r^{0}|P^{0})\,d\nu_2(P^{0})\quad\forall A \in \mathcal{B}.
\end{equation*}
\noindent At this point, our aim is to show that $\forall \varepsilon > 0$ we can construct a compact set $C_{\varepsilon}$ such that $\Theta_{t} \nu(\Omega \times C_{\varepsilon}) > 1 - \varepsilon$. We define $C_{\varepsilon}:=K_1 \times K_2^{\varepsilon}$, where $K_1 \subset \mathcal{X}_1$ and $K_2^{\varepsilon} \subset \mathcal{X}_2$ are two suitable compact sets. The set $K_1$ is the set defined in Equation~\eqref{eq:K_1}; in particular, Propositions~\ref{prop:bounded_solutions} and \ref{prop:uniform_persistence_for_N} ensure that $\varphi(t, \omega)N_r^{0} \in K_1\,\forall t \geq T_1\,\,\text{and}\,\,\forall \omega \in \Omega$. Therefore, we only have to find the set $K_2^{\varepsilon}$. Toward this aim, we recall the following two results, which are valid for the deterministic model in Equation~\eqref{eq:model_1}-\eqref{eq:model_3}.

\begin{proposition}[cfr.~\cite{Arl19}, Proposition 4.5]
\label{prop:price_dynamics}
Let $(N_r, P)$ be a solution of \eqref{eq:model_1}-\eqref{eq:model_3} and assume that some $t^* \geq t_0$ exists such that $0 < S_{\min} \le S(t) \le S_{\max}$ for all $t \ge t^*$, where we recall that $S_{\min}$ is defined in Proposition~\ref{prop:uniform_persistence_for_N} and $S_{\max}$ is defined in Proposition~\ref{prop:bounded_solutions}.
\begin{enumerate}
\item Let $P^* \geq 0$ be such that $D(P^*) < S_{\min}$. If $P(\hat{t}) > P^*$ for some $\hat{t} \ge t^*$, then we have
\[
P(t) < P^* \qquad \forall t > \hat{t} + \dfrac{P(\hat{t}) - P^*}{\lambda P^*} \cdot \dfrac{D_{0} + S_{\max}}{S_{\min} - D(P^*)}\,.
\]
\item Let $P_* \geq 0$ be such that $D(P_*) > S_{\max}$. If $0 < P(\hat{t}) < P_*$ for some $\hat{t} \ge t^*$, then we have
\[
P(t) > P_* \qquad \forall t > \hat{t} + \dfrac{P_* - P(\hat{t})}{\lambda P(\hat{t})} \cdot \dfrac{D_{0} + S_{\max}}{D(P_*) - S_{\max}}\,.
\]
\end{enumerate}
\end{proposition}

\begin{corollary}[cfr.~\cite{Arl19}, Corollary 4.6]
\label{prop:price_dynamics_2}
Assume that $m_0R_0c_0>2$, $D_{0}>S_{\max}$ and $S_{\min} > 0$ and let $C>1$ and $P_{\min}, P_{\max}$ be such that $[P_{\min}, P_{\max}] = D^{-1}([S_{\min}, S_{\max}])$.
\begin{enumerate}
\item If $ S(t) \ge S_{\min} $ for all $t \ge t^*$ and  $P(\hat{t}) \in [P_{\min} / C, C P_{\max}])$ for some $\hat{t} \ge t^*$, then $P(t) \in [P_{\min} / C, C P_{\max}])$ for all $t \ge \hat{t}$.
\item In any case, for every non trivial solution $(N_r, P)$ of \eqref{eq:model_1}-\eqref{eq:model_3} there exists $\hat{t}$ such that $P(t) \in [P_{\min} / C, C P_{\max}])$ for all $t \ge \hat{t}$.
\item Moreover, as long as $P(t)$ stays in $[P_{\min} / C, C P_{\max}])$, we have that $|P^{\prime}(t)| \le C \lambda P_{\max}$.\qedhere
\end{enumerate}
\end{corollary}

\noindent The previous results guarantee that there exists a $P^{*}$ such that $F(P(t), S(t))<0$ whenever $P(t)>P^{*}$. We define $K_2^{\varepsilon}$ as the set
\begin{equation*}
\label{eq:K2epsilon}
 K_2^{\varepsilon}:=\{P^{0}\in\mathcal{X}_2\,:\,\sup_{t \in [-T_1,0]}P^{0}(t)\leq P^{*}\,\text{for some}\,P^{*}\}
\end{equation*}
\noindent Without loss of generality, we can assume $P^{*}$ big enough in order to have that $\nu(K_2^{\varepsilon})>1-\varepsilon$. In particular, because of the mean-reversion behaviour of the stochastic price process, it is not difficult to prove the following 
\begin{proposition}\label{prop:bounded}
Let $P(0) < P^{*}$. Then, $\forall\varepsilon>0$ there exists an $M>0$ such that $\forall t>T_1$ we have that $\mathcal{P}(P(t) > M) < \varepsilon$.
\end{proposition}
\begin{proof}
We firstly recall that the drift term is smaller than $0$ (more precisely, it's smaller than $a<0$) for $P(t)>P^*.$ 
\noindent We introduce two sequences of stopping times to control the behaviour of the stochastic process $P(t)$:
\begin{align*}
\tau_0&:=0;\\
\sigma_1 &:= \inf\{ t \ge 0 : P(t) \ge P^* \};\\
\tau_n &:= \inf\{ t \ge \sigma_n : P(t) \ge P^*+1\};\\
\sigma_{n+1} &:= \inf\{ t \ge \tau_n : P(t) \le P^*\};
\end{align*}
which are nothing but the ``up-crossings'' at the levels $P^*$ and $P^*+1.$

\noindent We also define the index $N(t):= \inf \{ k : \tau_k >t\}-1$, and the stopping time $\tau_t=\tau_{N(t)}.$

\noindent We thus have the following inequalities, for any $M>P^*+1$:
\begin{align*}\mathbb{P}[P(t)>M]&=\underbrace{\mathbb{P}[P(t)>M,\tau_t=0]}_{=\,0}+\underbrace{\mathbb{P}[P(t)>M,\tau_t\ne0, P(t)<P^*+1]}_{=\,0}+\\ &\quad+\mathbb{P}[P(t)>M,\tau_t\ne0, P(t)>P^*+1]\\
&\le\mathbb{P}[P(t)-P(\tau_t)>M-P^*-1,\tau_t\ne0]\\
&\le \mathbb{P}[a(t-\tau_t)+W_t-W_{\tau_t}>M-P^*+1, \tau_t\ne0]\\
&\le \mathbb{P}[a(t-\tau_t)+W_t-W_{\tau_t}>M-P^*+1]
\end{align*}
 and the last term is smaller than $\epsilon$ for all $t>0$ if we pick $M$ big enough. The first inequality above is due to the fact that $P(\tau_t)$ is equal to $P^*+1,$ since $\tau_t$ is not zero. The second inequality is due to the fact that $P$ remains always above $P^*$ between $\tau_t$ and $t$: if it got below $P^*$ we would have another up-crossing, being $P(t)$ above $P^*+1,$ which would collide with the definition of $N(t)$. Since when $P(t)>P^*$ we have $P(t)-P(\tau_t)<a(t-\tau_t)+W_t-W_{\tau_t},$ we thus conclude the proof.
\end{proof}

In particular, $\forall\,t \geq T_1$ we have that: 
\begin{equation}\label{eq:bound}
    \Theta_t\nu\big(\Omega\times(K_1\times K_2^\eps)\big) \ge 1-\eps,
\end{equation}
\noindent which is translated into a bound on $\mu_t(\Omega \times C_{\varepsilon})$, being the latter the average of $\Theta_t\nu(\Omega \times  C_{\varepsilon})$. The following corollary immediately follows.

\begin{corollary}\label{corol:existence_random_measure}
Given the bound in Equation~\eqref{eq:bound}, the family $(\mu_t)_{t \geq T_1}$ is tight. As a consequence, there exists an invariant random measure $\boldsymbol\mu$ which is the limit in the weak topology of a sub-sequence of $(\mu_t)_{t \in T}$ for the model in Equations~\eqref{eq:model_1_prime}-\eqref{eq:model_3_prime}.
\end{corollary}
\begin{proof}
The existence follows from Theorem\ref{th:prohorov_random}, whereas the invariance from Proposition\ref{prop:averaging}.
\end{proof}

\noindent In addition, we can show the following proposition. 
\begin{proposition}
It holds that $\boldsymbol\mu\big(\Om\times (C_{\min}\times\Xx_2)\big)=0,$ where $$C_{\min}:=\Ll\{N\in\Xx_1:\min_{t\in[-T_1,0]}N(t)<N_{\min}\Rr\}.$$
\end{proposition}
\begin{proof}
Since $\boldsymbol\mu\big(\Om\times (C_{\min}\times\Xx_2)\big)$ is the limit of (a subsequence of) the averages of $\Theta_t\nu\big(\Om\times (C_{\min}\times\Xx_2)\big)$, it is sufficient to show that these quantities tend to zero as $t$ goes to $+\infty.$. Thanks to Proposition \ref{prop:uniform_persistence_for_N}, if $m_0R_0c_0>2$ we have that $$\phi(t,\om)^{-1}(C_{\min})\subset\Ll\{N_r^0\in\Xx_1: \min_{t\in[-T_1,0]}N(t)<\eta(t)\Rr\},$$
for some $\eta(t)$ decreasing function converging to $0$ as $t$ goes to infinity. In particular, the measure $\nu$ of this set converges to zero; indeed, by definition, $\nu$-almost surely the functions $N_r^0$ are positive almost everywhere on $[-T_1,0]$. Thus, $\Theta_t\nu\big(\Om\times (C_{\min}\times\Xx_2)\big)\to0$ as $t\to\infty.$ Thus, we can conclude that $\boldsymbol\mu\big(\Om\times(C_{\min}\times\Xx_2)\big)=0$ by approximating the indicator function of the set $\Om\times(C_{\min}\times\Xx_2)$ with bounded continuous functions $f$.
\end{proof}

\begin{proposition}
It holds that $\boldsymbol\mu\big(\Om\times (C_{\min}\times\Xx_2)\big)=0,$ where $$C_{\min}:=\Ll\{N\in\Xx_1:\min_{t\in[-T_1,0]}N(t)<N_{\min}\Rr\}.$$
\end{proposition}
\begin{proof}
Since $\boldsymbol\mu\big(\Om\times (C_{\min}\times\Xx_2)\big)$ is the limit of (a subsequence of) the averages of $\Theta_t\nu\big(\Om\times (C_{\min}\times\Xx_2)\big)$, it is sufficient to show that these quantities tend to zero as $t$ goes to $+\infty.$. Thanks to Proposition \ref{prop:uniform_persistence_for_N}, if $m_0R_0c_0>2$ we have that $$\phi(t,\om)^{-1}(C_{\min})\subset\Ll\{N_r^0\in\Xx_1: \min_{t\in[-T_1,0]}N(t)<\eta(t)\Rr\},$$
for some $\eta(t)$ decreasing function converging to $0$ as $t$ goes to infinity. In particular, the measure $\nu$ of this set converges to zero; indeed, by definition, $\nu$-almost surely the functions $N_r^0$ are positive almost everywhere on $[-T_1,0]$. Thus, $\Theta_t\nu\big(\Om\times (C_{\min}\times\Xx_2)\big)\to0$ as $t\to\infty.$ Thus, we can conclude that $\boldsymbol\mu\big(\Om\times(C_{\min}\times\Xx_2)\big)=0$ by approximating the indicator function of the set $\Om\times(C_{\min}\times\Xx_2)$ with bounded continuous functions $f$.
\end{proof}

\section{Numerical study of the deterministic model and of its dependence on the integration step}\label{sec:numerical}
In this section, we extend and complement the numerical study of the deterministic model given by Equations~\eqref{eq:model_1}--\eqref{eq:model_3} performed in \cite{Arl19}. First, we describe the discretization of the model. We will partition the time interval $\Delta t = 1$, which corresponds to one year, with both an integer and a non-integer number of steps $q$\footnote{In the first case, we will obtain a deterministic dynamical system in a phase-space of dimension $2 \times (T_1 \times q + 1)$. In the second case, instead, the dimension will be $2 \times (T_1\times \lceil q\rceil+1)$}; see Subsection \ref{subsec:discretization_of_the_model}. Second, we describe how we set both the initial conditions and the parameters of the model; see Subsections \ref{subsec:initial_conditions} and \ref{subsec:parameters}, respectively. Third, in Subsection \ref{subsec:study_one} we focus on our main objective for this Section and we investigate the dependence of the attractor on the number of integration steps per time unit for two different sets of experimental parameters. The main result of this study is that the chaotic attractor is persistent under the dimensionality reduction process associated to choosing longer and longer integration time steps. In order to make this claim more precise and less qualitative, 
in Subsection \ref{subsec:distances} we introduce several statistical notions of distance between 
attractors and we numerically compute them for quantifying the distance of the "asymptotic" attractor (i.e.\ corresponding to the continuous time system, approximated by a very short integration time) and the one corresponding to the choice of a relatively large integration step.

\subsection{Discretization of the model}\label{subsec:discretization_of_the_model}
The phase-space of the original model is infinite dimensional. Here we describe how to numerically integrate it by using an integer number of steps $q$. The time parameter $t \in [0,+\infty)$ is replaced by indices $i \in \mathbb{N} \setminus \{0\}$; roughly, $i \geq 1$ replaces the interval $\left(\frac{i-1}{q}, \frac{i}{q}\right]$. We now explain the notation we are going to use; notice that the discretized values of the different functions at index $i$ do not correspond to their values on $t=\frac iq$ but, instead, to their integral or their average on the interval $\left( \frac{i-1}{q} , \frac{i}{q} \right]$.

\begin{enumerate}[label=(D\arabic*)]
\item\label{itm:D1} $N_{r,i} $ approximates the average of $N_r(t)$ over $t \in \left( \frac{i-1}{q} , \frac{i}{q} \right]$: $N_{r,i} \approx q \int_{\frac{i-1}{q}}^{\frac{i}{q}} N_r(t) \diff t$. Similarly, $N_{b,i}$, $\rhoseasondiscrete{i}$, $S_i$ and $P_i$ are constructed from $N_{b}(t)$, $\rhoseason (t)$, $S(t)$ and $P(t)$ as their average over the interval $t \in \left( \frac{i-1}{q} , \frac{i}{q} \right]$.
\item\label{itm:D2} $B_{r,i}$ approximates the number of newborn females put in the reproducing line 
in the interval $t \in \left( \frac{i-1}{q} , \frac{i}{q} \right]$: $B_{r,i} \approx \int_{\frac{i-1}{q}}^{\frac{i}{q}} B_r(t) \diff t$. The fact that this quantity is not divided by $q$ - differently from $N_r(t)$ - is because $B_r(t)$ is a density (number per unit of time), whereas $N_r(t)$ is the number of individuals.\\
\item The quantities $B_{b,i}$, $B_{f,i}$, $B_{m,i}$ are defined in the same fashion, respectively from $B_{b}(t)$, $B_{f}(t)$, $B_{m}(t)$.
\end{enumerate}

\noindent We construct the discretized model at any step $k \geq 1$ by computing 
\begin{equation}\label{eq:Z_k}
    Z_k := (N_{r,k} \, ; \, N_{b,k} \, ; \, S_k \, ; \, P_k \, ; \, B_{r,k} \, ; \, B_{b,k})
\end{equation}
from its past values 
$$ \paren{Z_i}_{ k - \max\{k_{A_1},k_{\Omega_1} \} \leq i \leq k-1},$$ 
in the following order:
\begin{align}
N_{r,k} &= \sum_{j= k_{A_0}}^{k_{A_1}} B_{r,k-j}\label{eq:discretized_N_r_k},\\
N_{b,k} &= \sum_{j= k_{\Omega_0}}^{k_{\Omega_1}} B_{b,k-j}\label{eq:discretized_N_b_k},\\
S_k     &= \frac{q N_{b,k}}{k_{\Omega_1}-k_{\Omega_0}+1}\label{eq:discretized_S_k},\\
P_k     &= \max\Ll\{ 0, P_{k-1} + \frac{\lambda P_{k-1} F \bigl( D(P_{k-1}),S_k \bigr)}{q} \Rr\}\label{eq:discretized_P_k},\\ 
B_{r,k} &= \frac{m_0}{q} \rhoseasondiscrete{k} m(N_{r,k}) R(P_k)\label{eq:discretized_B_r_k},\\
B_{b,k} &= \frac{m_0}{q} \rhoseasondiscrete{k} m(N_{r,k}) \bigl(2 - R(P_k)\bigr)\label{eq:discretized_B_b_k}.
\end{align}

\noindent Notice that the Equations~\eqref{eq:discretized_N_r_k}-\eqref{eq:discretized_N_b_k} and \eqref{eq:discretized_S_k} depend on the parameters $k_{A_0},k_{A_1},k_{\Om_0},k_{\Om_1}$: these are the discretized version of the parameters $A_0,A_1,\Om_0,\Om_1$ of the original dynamical system, but with some minor technical modifications needed to ensure the convergence of the numerical scheme (and that the discretization leads to a finite-dimensional dynamical system). Namely, we define:
\begin{enumerate}
\item $k_{A_0}=\max\{1, \croch{q A_0} \}$, where $\croch{x}$ is the closest integer to $x$. Notice that the  
$\max\{1,\cdot\}$ is because we want $k_{A_0} >0$ so that $N_{r,k}$ only depends 
on the past of $B_r$ in Equation~\eqref{eq:discretized_N_r_k}, otherwise we would have a circular definition.
\item $k_{A_1}=\max\{k_{A_0},\croch{q A_1}-1\}$. We take $\croch{q A_1}-1$ instead of $\croch{q A_1}$ because animals are supposed to die exactly at age $A_1$, so individuals of age $A_1$ at time $k/q$ (counted in $B_{r , k-\croch{q A_1}}$) do not count in $N_{r,k}$. 
At the same time, we want $k_{A_1}\geq k_{A_0}$ in order to make the sums defining $N_{r,k}$ and $N_{b,k}$ non-empty when $A_0$ is really close to $A_1$. 
\item $k_{\Omega_0}=\max\{1, \croch{q \Omega_0} \}$ is defined similarly to $k_{A_0}$.
\item $k_{\Omega_1}=\max\{k_{\Omega_0},\croch{q \Omega_1}-1\}$ is defined similarly to $k_{A_1}$.
\end{enumerate}
\paragraph{Non-integer number of steps} When $q\notin \N$, we partition in $\lceil q\rceil$ parts each interval $[n,n+1]$ for $n\in\N$. The first $\lfloor q\rfloor$ parts are long $\frac1{q},$ whereas the last one is long $\frac{\{q\}}{q},$ where $\{q\}$ denotes the fractional part of the number $q$. An example of this kind of partition for $q=5.5$ is shown in the following picture:
\scalebox{1.4}{
\begin{tikzpicture}
\node[scale=0.65] at (3,0.8)  {};
\node[scale=0.65] at (0,0.5)  {$i=0$};
\node[scale=0.65] at (1,0.5)  {$i=1$};
\node[scale=0.65] at (2,0.5)  {$i=2$};
\node[scale=0.65] at (3,0.5) {$i=3$};
\node[scale=0.65] at (4,0.5) {$i=4$};
\node[scale=0.65] at (4.9,0.5) {$i=5$};
\node[scale=0.65] at (5.65,0.5) {$i=6$};
\node[scale=0.65] at (0,-0.5)  {$t=0$};
\node[scale=0.65] at (1,-0.5)  {$t=\frac{1}{5.5}$};
\node[scale=0.65] at (2,-0.5)  {$t=\frac{2}{5.5}$};
\node[scale=0.65] at (3,-0.5) {$t=\frac{3}{5.5}$};
\node[scale=0.65] at (4,-0.5) {$t=\frac{4}{5.5}$};
\node[scale=0.65] at (4.9,-0.5) {$t=\frac{5}{5.5}$};
\node[scale=0.65] at (5.65,-0.5) {$t=1$};
\node[scale=1] at (-1,0)(L){};
\node[scale=1] at (-1.5,-0.013){$\cdots$};
\node[scale=1] at (6.5,0)(R){};
\node[scale=1] at (7,-0.013){$\cdots$};
\draw[->] (L)--(R) node [midway,above] {};
\draw (0,0.2) -- (0,-0.2) node [midway, above, sloped]{};
\draw (1,0.2) -- (1,-0.2) node [midway, above, sloped]{};
\draw (2,0.2) -- (2,-0.2) node [midway, above, sloped]{};
\draw (3,0.2) -- (3,-0.2) node [midway, above, sloped]{};
\draw (4,0.2) -- (4,-0.2) node [midway, above, sloped]{};
\draw (5,0.2) -- (5,-0.2) node [midway, above, sloped]{};
\draw (5.5,0.2) -- (5.5,-0.2) node [midway, above, sloped]{};
\end{tikzpicture}}
\noindent Also, it is necessary to modify the discretization of the functions $N_r, N_b, S, P$ and $m_\rho$ in \ref{itm:D1}. Indeed, for $i\neq 0 \mod \lceil q\rceil$ the quantity $N_{r,i}$ --and analogously $N_{b,i},S_{i},P_{i},m_{\rho,i}$-- is defined as the approximation of the following quantity: 
\begin{equation*}
    N_{r,i} \approx q \int_{t_i-\frac1q}^{t_i} N_r(t)\diff t,\quad t_i:=\left\lfloor \frac{i}{\lceil q\rceil}\right\rfloor+\left(i-\left\lfloor \frac{i}{\lceil q\rceil}\right\rfloor\right)\cdot\frac1q.
\end{equation*}
Instead, $i =0 \mod\lceil q\rceil$ the quantity $N_{r_i}$ approximates the following integral:
\begin{equation*}
    N_{r,i} \approx \frac q{\{q\}} \int_{t_i-\frac{\{q\}}q}^{t_i} N_r(t) \diff t,\quad t_i:=\left\lfloor \frac{i}{\lceil q\rceil}\right\rfloor.
\end{equation*}
\noindent Finally, in order to satisfy this approximation requirement, we re-weight the components of $Z_k$ in Equation~\eqref{eq:Z_k} dividing them by $\{q\})$ whenever $k$ is a multiple of $\lceil q\rceil$.

\subsection{Initial conditions}\label{subsec:initial_conditions}
The initial conditions are given by the values of the number of newborns ($B_{r,i}$ and $B_{b,i}$) 
for time steps $i=1,\ldots, k_{A_1}$ before the start of the simulation, since the discretized model in Subsection \ref{subsec:discretization_of_the_model} only needs the past values of $B_r$ and $B_b$ to compute the future dynamics of the model. For each parameter set, we choose an initial condition randomly as follows: $B_{r,1} ,\ldots, B_{r,k_{A_1}} , B_{b,1} , \ldots, B_{b,k_{A_1}} $  are chosen independently with common distribution $ \mathcal{U}\paren{ \croch{ 0 , \frac{2}{k_{A_1} - k_{A_0} + 1} }} $.
The reason for this choice is that $k_{A_1} - k_{A_0} + 1$ is the number of time steps 
corresponding to reproducing ages of females. So, for instance, the reproducing female population at time $k_{A_1}+1$ is the average of random variables uniform over $[0,2]$, so it should be close to 1, the threshold value for density-dependence to apply.

\subsection{Parameters of the model}\label{subsec:parameters}
We follow \cite{Arl19}, and we fix the functions $m, m_\rho,D$ and $R$ in such a way that they have to satisfy the requirements of Section \ref{sec:first_analysys}. In particular in our experiments we make the following choices:
\begin{enumerate}[label=(F\arabic*)]
\item\label{itm:F1}\textbf{The fertility function.} We take:
$$m(N) = m_0 \bigl( \max \{N,1\} \bigr)^{-\gamma},$$
where $\gamma$ and $m_0$ are parameters of the model to be chosen.
\item\label{itm:F2}\textbf{Seasonality.} We let
$$\rhoseason(t) = \frac{1}{1-\rho} \mathsf{1}_{\{t - \lfloor t \rfloor \in [0,1-\rhowinter)}\}.$$
In this way, we model the fact that the births are synchronized and concentrated in one specific period of the year.
\item\label{itm:F3} \textbf{Demand function.} We choose the following demand function 
\begin{equation*} \label{eq:demand_function}
D(P) = D_{\exp}(P) := D_0 e^{ -\alpha_D P } 
\end{equation*}
for some parameters $D_0,\alpha_D>0$.
\item\label{itm:F4} 
\textbf{Breeder strategy:} We assume that the breeder follows a counter-cyclical policy and takes
$R$ close to 1 when the price is high. In the numerical experiments we choose 
\begin{equation}\label{eq:breeder_logistic} 
    R(P) = R_{\mathrm{logistic}}(P) := R_0 + (R_1-R_0) f_d(P/P_0) 
\end{equation}
where 
\[
f_d(x) = \begin{cases}
\frac{x^d}{2}                        &\qquad \mbox{if } x \in [0,1) \\
\frac{1}{1 + \exp\paren{-2 d (x-1)}} &\qquad \mbox{otherwise} \end{cases}
\]
and $R_0, R_1 \in [0,1]$ and $P_0, d > 0$ have to be chosen. This is a sigmoid function between $0$ and $+\infty$.
\end{enumerate}
Finally, for the sake of simplicity, we took the parameters $\Omega_{0}$ and $\Omega_{1}$ in the supply function to be equal to $A_{0}$ and $A_{1}$,
respectively. The normalization of the supply is $\Delta\Omega=\Omega_1-\Omega_0=A_1-A_0$.

\subsection{Dependence of the attractor on the discretization time step $q$}\label{subsec:study_one}
The aim of this section is to study the dimensionality reduction of the dynamical system in Equations~\eqref{eq:model_1}--\eqref{eq:model_3} as a function of the number $q$ of steps per year. In order to do so, we visualize attractor of the deterministic model by plotting in $\mathbb{R}^{3}$ the set $C(t) := (N_r(t), N_r(t-1), N_r(t-2)$, $t \in \mathbb{N}$, for different values of $q$. The goal will be to see how much the geometric characteristics of the attractor will change by passing from, e.g., $q=100$ to, e.g., $q = 2$. Should the chaotic dynamics be maintained, this would be a remarkable result. Indeed we would pass from a dynamical system with a dimension of the phase-space equal to $200$ to a dynamical system with a dimension of the phase-space equal to $8$ and, following \cite{YB98}, for $t \in [290000, 300000]$. We define the first set of experimental parameters, say $\mathcal{H}_1$, in the following way:
\begin{enumerate}[label=(S.1.\arabic*)]
\item\label{itm:S1} \textbf{Population dynamics:} $A_0= 0.18$, $A_2= 2.0$, $m_0= 5.0$, $\gamma=8.25$ and $\rho = 0.79$.
\item\label{itm:S2} \textbf{Market dynamics:} $\lambda = 1$ and the demand function is $D=D_{\exp}$ with $D_0 = 5$ and $\alpha_D =1$.
\item\label{itm:S3} \textbf{Interaction between population and market:} $\Omega_0=0.18$, $\Omega_1=2$, and $R=R_{\mathrm{logistic}}$ with $R_0 = 0$ (\emph{minimal value}), $R_1=1$ (\emph{maximal value}), $P_0=1$ (\emph{price threshold}) and $d=4$ (\emph{degree of $R(P)$ for small $P$}).
\end{enumerate}
\noindent This set is called SP in \cite{Arl19} and is close to the main setting for the Birkeland-Yoccoz model numerically solved in \cite{Arl04}. The left (resp.~right) panel of Figure~\ref{fig:attractors_1} displays the dynamics of the set $C(t)$ associated to $q=100$ (resp.~to $q=2$) steps per year. Interestingly, several skeleton of the attractor and some of its structure survives. On the other hand, the geometric structure of the set $C(t)$ is undoubtedly simplified with respect to the set depicted on the left-panel of the same figure. Also, it presents some --a priori unexplained-- stronger accumulation towards the value of zero.

some of the geometrical features are preserved; indeed, even with $q=2$ there are still some of the spikes and of the edges of the three dimensional projection with $q=100$. 

\begin{figure}[h]
\centering
\includegraphics[width=0.495\linewidth]{fig1.png}
\includegraphics[width=0.495\linewidth]{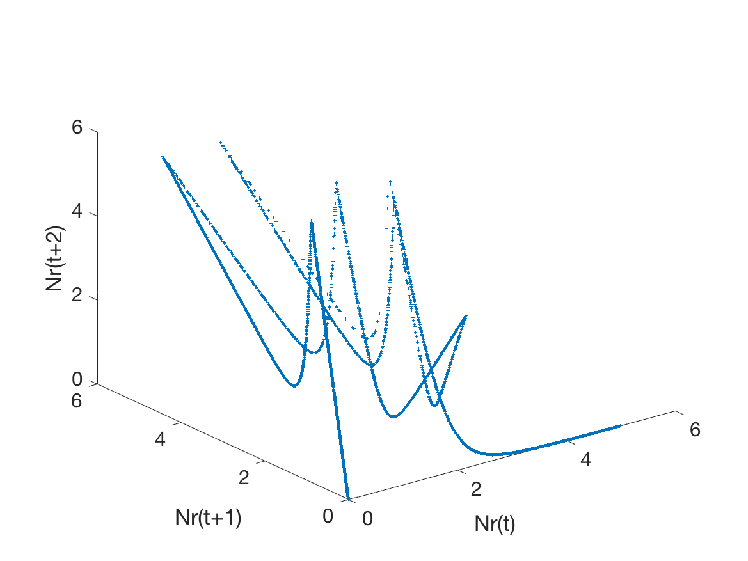}
\caption{Three dimensional plot of $C(t)=(N_r(t), N_r(t+1), N_r(t+2))$ of the deterministic dynamical system in Equations~\eqref{eq:model_1}--\eqref{eq:model_3}. Setting $\mathcal{H}_1$. \emph{Left Panel:} $q = 100$. \emph{Right Panel:} $q = 2$.}
\label{fig:attractors_1}
\end{figure}

Actually, the dimension of the dynamical system's phase-space with $q=2$ can be lowered to 5; indeed, the following lemma holds true.
\begin{lemma}
The discretized dynamics that arise by choosing $q=2$ with the set of experimental parameters $\Hh_1$ below can be modeled through a dynamical system whose phase-space dimension is equal to $5$.
\end{lemma}
\begin{proof}
First, we notice that $m_{\rho,k}$ consists of a sequence of $2$'s for $k$ even and of zeros for $k$ odd. Also, we point out that the values of $k_{A_0}$ and $k_{\Om_0}$ are $1$ and the values of $k_{A_1}$ and $k_{\Om_1}$ are $3$. As a consequence, in the summations defining $N_{r,k}$ and $S_k$ (see Equations~\eqref{eq:discretized_N_r_k} ans \eqref{eq:discretized_S_k}) the terms $N_{r,i}$ and $P_i$ for $i$ odd do not matter, since they are multiplied by $m_{\rho,i}$ which is zero for $k$ odd. Also, $P_k$ in Equation~\eqref{eq:discretized_P_k} depends only on $P_{k-1}$ and on $S_k,$ which in turn depends on the values of $N_{r,i}$ and $P_i$ for even $i$. Therefore, for $q=2$, we can construct a dynamical system consisting only of the couple $(N_{r,k},N_{r,k-2})$ and of the prices $(P_k,P_{k-1},P_{k-2})$ whose evolution is given by
$$\binom{N_{r,k-2},\,N_{r,k}}{P_{k-2},\,P_{k-1},\,P_k}\mapsto \binom{N_{r,k},\,\frac{m_0}qm_{\rho,k}m\big(N_{r,k}R(P_k)\big)}{P_k,\,0\vee P_{k}+\frac1q\big(\lambda P_{k}F(F(P_{k}),S_{k+1})\big),\,P_{k+2}}.$$
Notice that in the previous equation we pointed out that $S_{k+1}$ can be obtained from $N_{r,k},N_{r,k-2},P_k$ and $P_{k-2}$, and that $P_{k+2}$ can be obtained from the value of $$P_{k+1}=0\vee P_{k-1}+\frac1q\big(\lambda P_{k-1}F(D(P_{k-1}),S_k)\big)$$ and from that of $S_{k+2}$, which can be computed from $N_{r,k}$ and $P_k.$
\end{proof}

\indent A more interesting phenomenon is shown in Figure~\ref{fig:attractors_2}. The left panel displays the dynamics of the set $C(t)$ associated to $q=10$ steps per year: as expected, the geometry of the set stands between the one with $q=2$ and $q=100$. However, the right panel shows the possibility that $C(t)$ exhibits, for some values of $q$, only a low-complexity seemingly non-chaotic orbit, very close to being periodic and totally different also from the original Yoccoz-Birkeland attractor. In particular, the bifurcation diagram of $N_r(t)$, $290000 \leq t \leq 300000$, as a function of the time discretization parameter $q \in [1, 40]$ in Figure~\ref{fig:bifurcation} shows that the presence of a quasi-periodic orbit is more frequent for small values of $q$.

\begin{figure}[!h]
\centering
\includegraphics[width=0.495\linewidth]{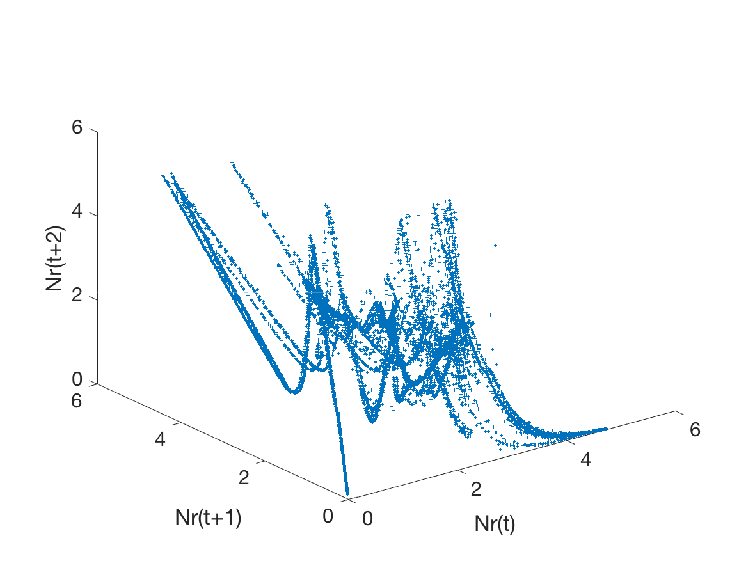}
\includegraphics[width=0.495\linewidth]{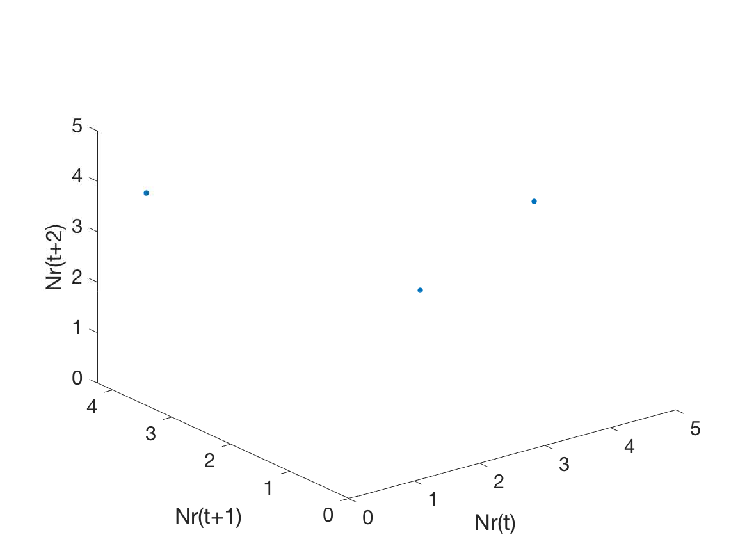}
\caption{Three dimensional plot of $C(t)=(N_r(t), N_r(t+1), N_r(t+2))$ of the deterministic dynamical system in Equations~\eqref{eq:model_1}--\eqref{eq:model_3}. Setting $\mathcal{H}_1$. \emph{Left Panel:} $q = 10$. \emph{Right Panel:} $q = 20$.}
\label{fig:attractors_2}
\end{figure}

\begin{figure}[!h]
\centering
\includegraphics[width=1\linewidth]{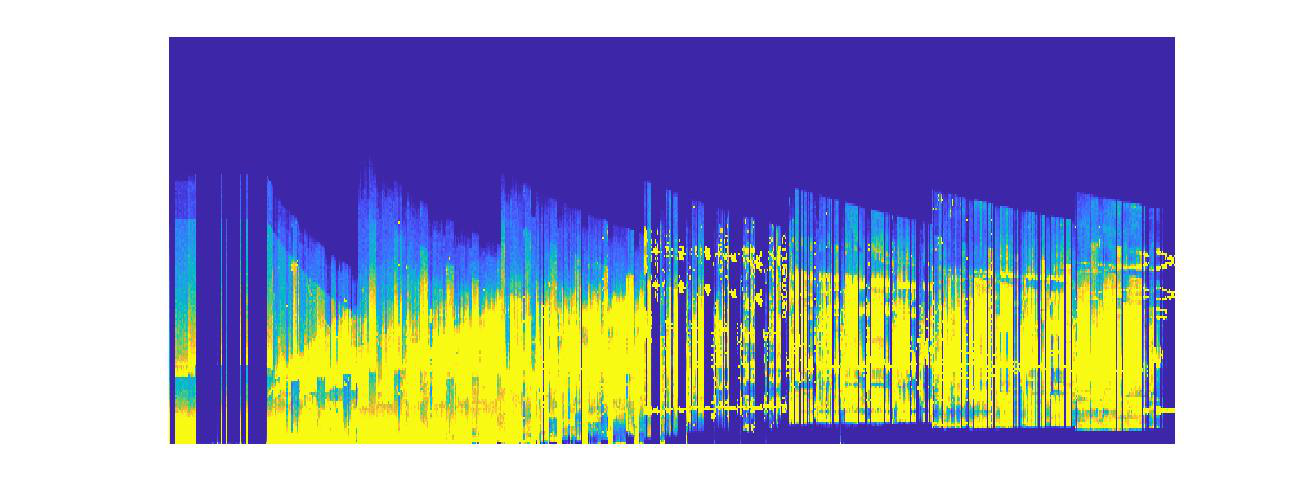}
\caption{Bifurcation diagram for $N_r(t)$, $290000 \leq t \leq 300000$, as a function of the parameter $q \in [1, 40]$. Setting $\mathcal{H}_1$.}
\label{fig:bifurcation}
\end{figure}

\subsubsection{Comparison with setting $\mathcal{H}_1$: small values of $q$ with a chaotic dynamics}\label{subsec:study_two}
We tune the set of experimental parameters $\mathcal{H}_1$ until the three-dimensional plots of $(N_r(t), N_r(t+1), N_r(t+2))$ for small and high values of $q$ look very similar. In particular, we name $\mathcal{H}_2$ this set of new experimental parameters. With respect to $\mathcal{H}_1$, we modify the following parameters : 
\begin{enumerate}[label=(S.2.\arabic*)]
\item\label{itm:S11} $A_0 = \Omega_0 = 0.8$, which corresponds to a longer period of growth of the cattle before they reach fertility.
\item\label{itm:S22} $m_0=10$. 
\item\label{itm:S33} $\gamma=7$.
\item\label{itm:S44} The degree $d$ of the $R_{\mathrm{logistic}}$ function is set equal to $2$.
\item\label{itm:S55} The coefficient $\alpha_D$ of the $D$ function is set equal to $2$.
\end{enumerate}

The left (resp.~right) panel of Figure~\ref{fig:attractors_3} displays the dynamics of the set $C(t)$ associated to $q=3$ (resp.~to $q=100$) steps per year in the set of experimental parameters $\mathcal{H}_2$. We make the following considerations. First, the accumulation towards the value of zero is preserved for small values of $q$. Second, even for such a values there exists a non trivial geometric structure that is similar to the one observed for $q=100$. Also, we find that for only a few values of $q$ the set $C(t)$ exhibits an orbit very close to being periodic; see Figure~\ref{fig:attractors_4}, left panel, for a graphical representation of this statement. As an example, Figure~\ref{fig:attractors_4}, right panel, displays the quasi-periodic orbit when $q=10.5$.   

\begin{figure}[h]
\centering
\includegraphics[width=0.495\linewidth]{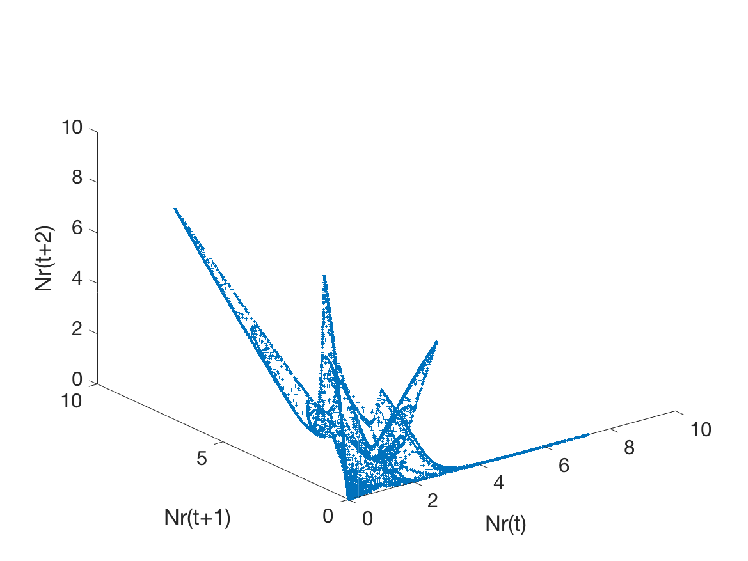}
\includegraphics[width=0.495\linewidth]{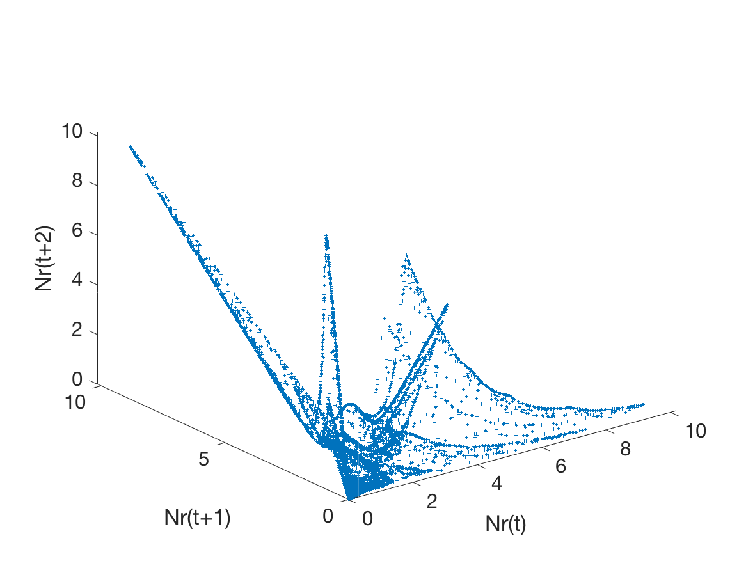}
\caption{Three dimensional plot of $C(t)=(N_r(t), N_r(t+1), N_r(t+2))$ of the deterministic dynamical system in Equations~\eqref{eq:model_1}--\eqref{eq:model_3}. Setting $\mathcal{H}_2$. \emph{Left Panel:} $q = 2$. \emph{Right Panel:} $q = 100$.}
\label{fig:attractors_3}
\end{figure}

\begin{figure}[!h]
\centering
\includegraphics[width=0.495\linewidth]{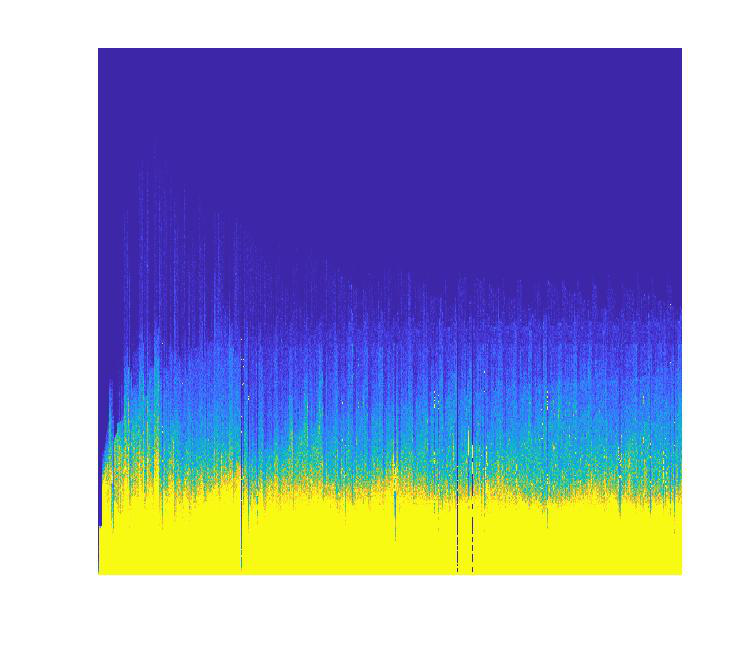}
\includegraphics[width=0.495\linewidth]{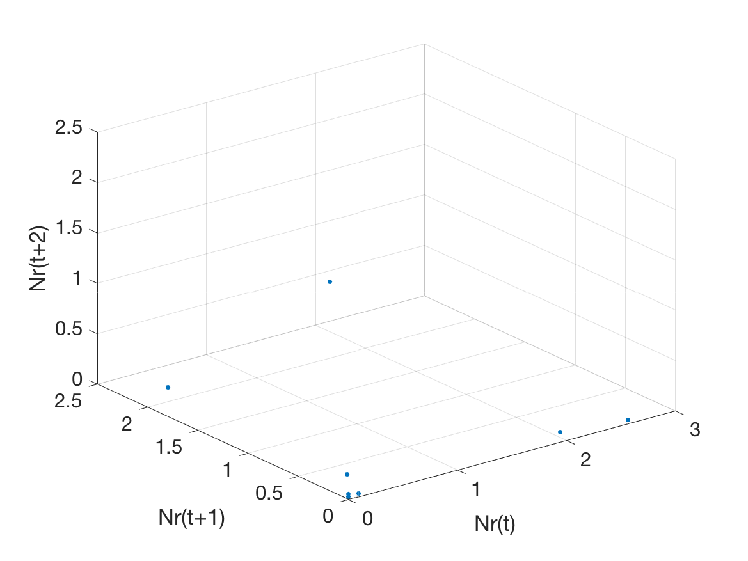}
\caption{\emph{Left Panel:} Bifurcation diagram for $N_r(t)$, $290000 \leq t \leq 300000$, as a function of the parameter $q \in [1, 40]$. Setting $\mathcal{H}_2$. \emph{Right Panel:} Three dimensional plot of $(N_r(t), N_r(t+1), N_r(t+2))$ of the deterministic dynamical system in Equations~\eqref{eq:model_1}--\eqref{eq:model_3}. Setting $\mathcal{H}_2$ and $q = 10.5$.}
\label{fig:attractors_4}
\end{figure}

\subsection{Measuring distances between attractors}\label{subsec:distances}
By comparing Figure~\ref{fig:attractors_1} with Figure~\ref{fig:attractors_3}, it is visually clear that the reduction in the number of integration steps $q$ distorts less the geometry of $C(t)$ in the case of the set of experimental parameters $\mathcal{H}_2$. In order to make this statement more rigorous, we use a number of \emph{Statistical Distances} proposed in the literature, i.e. a number of distances between measures defined on a metric space $(X,\delta )$. In particular, each distance is approximated by the distance between two histograms (i.e. between two vectors of the same length with positive entries that sum-up to one) constructed from the corresponding attractors. Precisely, we construct the histograms by measuring the density of the points constituting the plot of the attractors over a three dimensional grid $X$ of dimension $n_g \times n_g\times n_g:$ for computational reasons, we used $n_g=23$. Henceforth, we denote by $r$ the histogram associated to the attractor obtained with $q=100$ and we call this attractor \emph{asymptotic} attractor, and by $c_i$, $i = 1, \ldots, 100$, the histograms associated to the attractors obtained with $q \in [2, 100]$, on a pseudo-logarithmic scale. We use the following statistical distances between two given histograms $P$ and $Q$: 
\begin{enumerate}[label=(D\arabic*)]
\item\label{itm:D1} Kullback-Leibler divergence (\cite{KL51}). It is defined as
\begin{equation*}
    D_{KL}(P || Q) := \sum_{x \in X} P(x)\log\left(\frac{P(x)}{Q(x)}\right).
\end{equation*}
\item\label{itm:D2} Jensen-Shannon distance (\cite{WYG85}). It is defined as the smoothed and symmetrized version of the Kullback-Leibler, i.e.:
\begin{equation*}
    D_{JS}(P || Q) := \frac{1}{2}\left(D_{KL}(P || M) + D_{KL}(Q || M) \right)\quad\text{with}\quad M = \frac{1}{2}\left(P + Q \right).
\end{equation*}
\item\label{itm:D3} The Wasserstein distance (\cite{AMB05}). It is defined as:
\begin{equation*}
    W(P, Q):= \left(\inf_{\gamma \in \Gamma(P,Q)} \int_{X \times X} \delta (x, y)^2\,d\gamma(x, y)\right)^{1/2},
\end{equation*}
where $\Gamma(P, Q)$ denotes the collection of all measures on $X \times X$ with marginal distributions $P$ and $Q$ on the first and second factors respectively. 
\item\label{itm:D4} A regularized version of the Sinkhorn distance proposed by \cite{CU13}. It is defined in the following way. Let $r$ and $c$ be two given histograms in the simplex $\Sigma_d:= \{x\in \R^d_+: x^T\mathsf{1}_d=1\}$. Let $U(r,c)$ be the transportation polytope of $r$ and $c$, namely the polyhedral set of $d\times d$ matrices
$$U(r,c):= \{P\in\R_+^{d\times d}\; |\; P\mathsf{1}_d=r, P^T\mathsf{1}_d=c\},$$
where $\mathsf{1}_d$ is the $d$ dimensional vector of ones. 
Given a $d\times d$ cost matrix $M$, the cost of mapping $r$ to $c$ using a transportation matrix $P$ can be quantified as $\langle P,M\rangle, := \textrm{Tr}\, (P^TM)$ the Frobenius product between $P$ and $M$. The following minimization problem:
$$
D_M(r,c) := \min_{P\in U(r,c)} \langle P,M\rangle
$$
is called the \emph{optimal transportation} problem between $r$ and $c$, given the cost $M$, whereas $$
D^\lambda_M(r,c) := \min_{P\in U(r,c)} \langle P,M\rangle-\frac1\lambda h(P)
$$
is an approximated version of the problem above, where $h$ is the Shannon entropy, and $D^\lambda_M(r,c)$ is called {\it Sinkhorn distance}. Both $D_{M}$ and $D^\lambda_M$ are distances between two histograms; see \cite{CU13}. Last but not least, \cite{CU13} proposes a computationally efficient algorithm --called Sinkhorn-Knopp
Algorithm-- to compute the distance $D_M^{\lambda}$ between a given histogram $r$ and a set of histograms $\{c_i\}_{i \in I}$, as well as a lower bound on the real optimal transport distance.
\end{enumerate}

\noindent We discuss now the results. Figure~\ref{fig:distances_1}, left (resp.~right) panel, plots the distances just presented for the set of experimental parameters $\mathcal{H}_1$ (resp.~$\mathcal{H}_2$). For both settings, we notice that the distance is a decreasing function of $q$, correspondingly to the qualitative assessment that for higher values of $q$ we obtain figures that are more and more similar to the asymptotic one. However, we sometimes observe some spikes; they correspond to the values of $q$ for which the orbit is quasi-periodic. In addition, notice that these spikes are more sporadic in the setting $\mathcal{H}_2$, thus confirming what we observed in Section~\ref{subsec:study_one}. Moreover, the Sinkhorn distance (i.e. the transportation cost) for low values of $q$ is much higher for the setting $\mathcal{H}_1$ than for the setting $\mathcal{H}_2$. This confirms the visual insight that the latter setting is much more stable than $\mathcal{H}_1$ for small values of $q$, i.e. it leads to attractors which have a more gentle dependence on $q$. These considerations are also corroborated by Figure~\ref{fig:entropy_and_fractals}, which reports, as function of $q$, the entropy and the fractal dimension computed following the same methodology as in \cite{Arl19}: we postpone the discussion of these quantities to Section~\ref{subsec:entropy_and_fractal_random}. Finally, we remark that we also test the reliability of the previous methodology on the logistic map; see Appendix~\ref{app:logistic_attractors}.  

\begin{figure}[!h]
\includegraphics[width=0.50\textwidth]{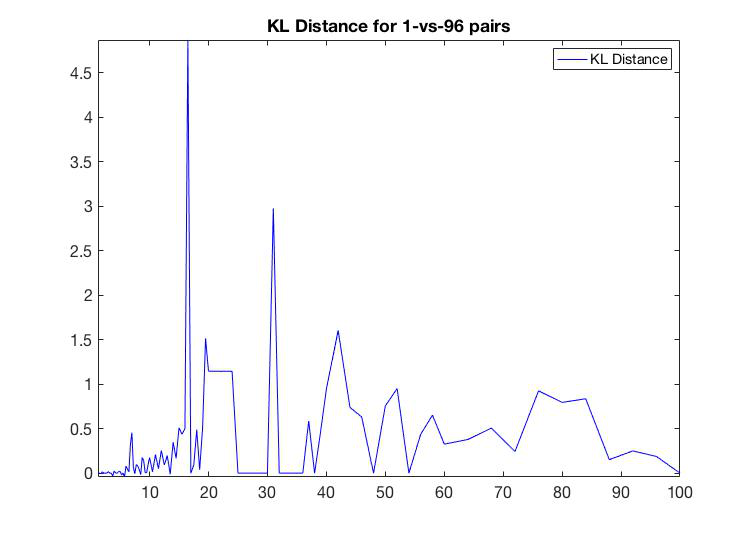}
\includegraphics[width=0.50\textwidth]{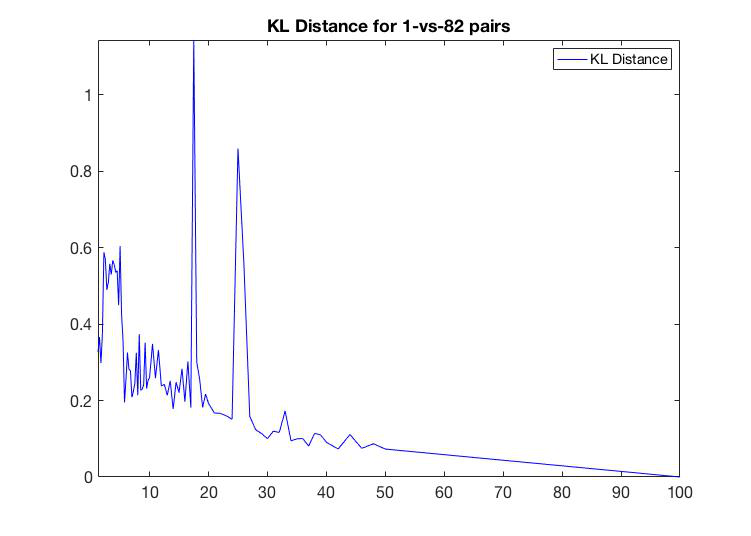}
\includegraphics[width=0.50\textwidth]{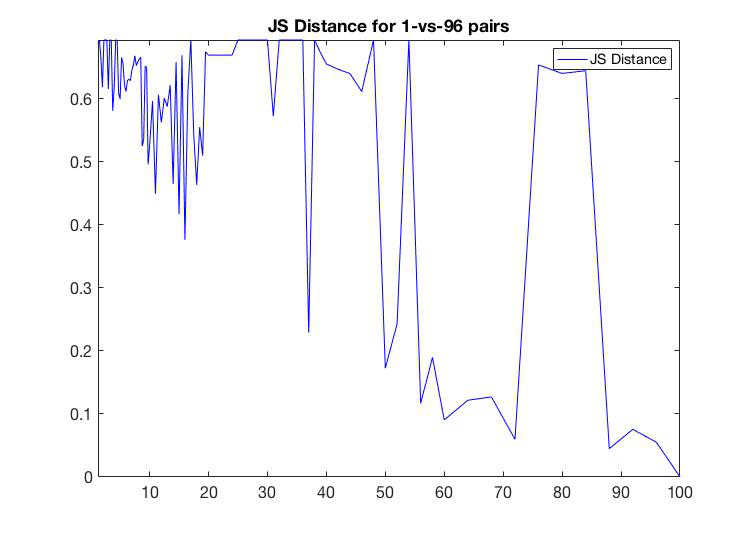}
\includegraphics[width=0.50\textwidth]{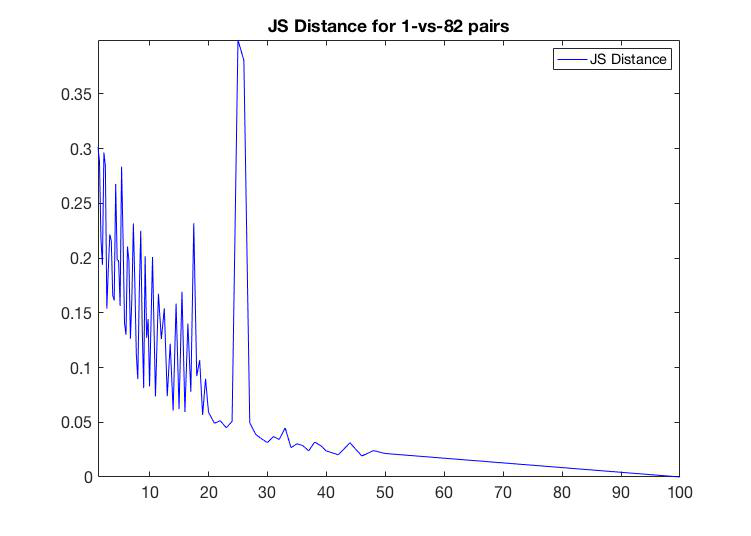}
\includegraphics[width=0.50\linewidth]{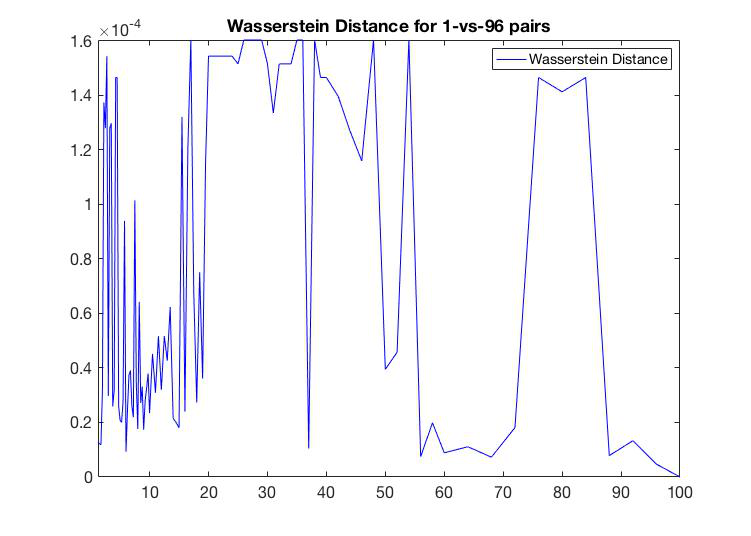}
\includegraphics[width=0.50\linewidth]{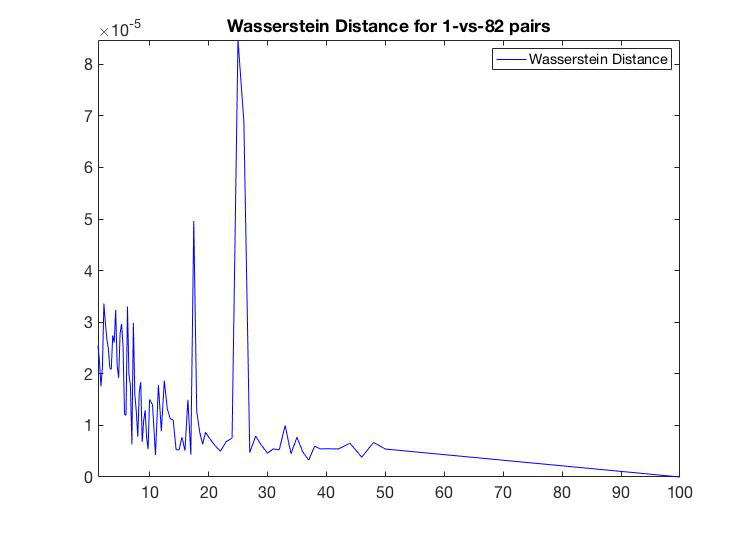}
\includegraphics[width=0.50\linewidth]{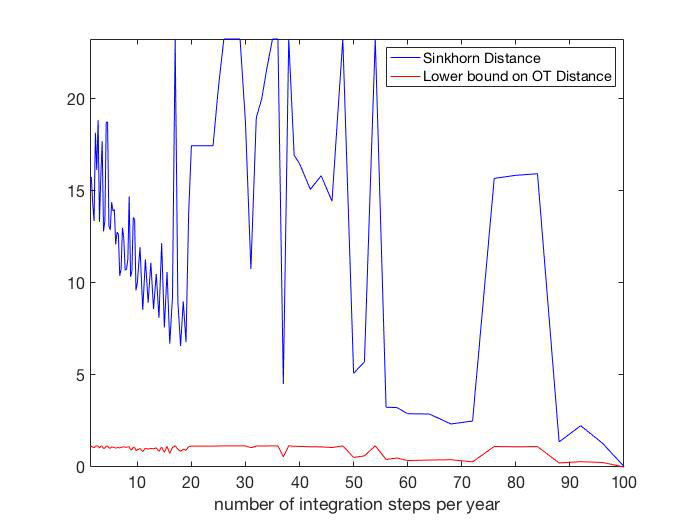}
\includegraphics[width=0.50\linewidth]{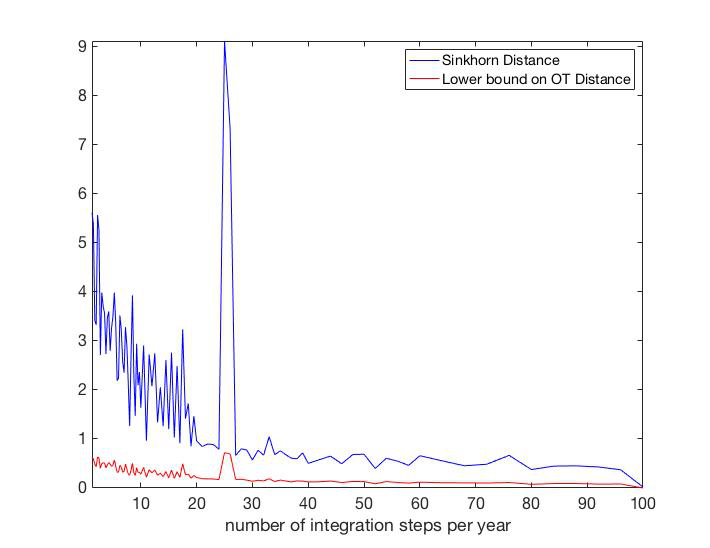}
\caption{\emph{From top to bottom, left panel:} distances between the attractors as a function of $q \in [2, 100]$ on a pseudo-logarithmic scale for the set of experimental parameters $\mathcal{H}_1$. \emph{From top to bottom, right panel:} the same quantities for the set of experimental parameters $\mathcal{H}_2$. The distances are, in order: the Kullback-Leibler divergence, the Jensen-Shannon distance, the Wasserstein distance and the Sinkhorn distance.}
\label{fig:distances_1}
\end{figure}

\begin{figure}[!h]
\includegraphics[width=0.52\textwidth]{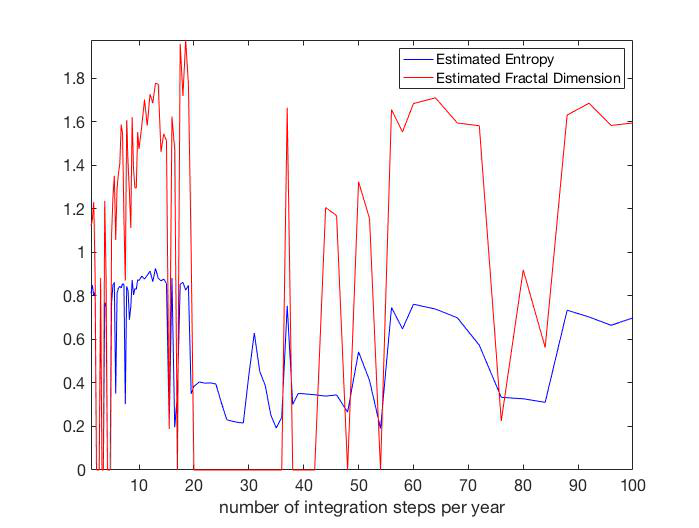}
\includegraphics[width=0.52\textwidth]{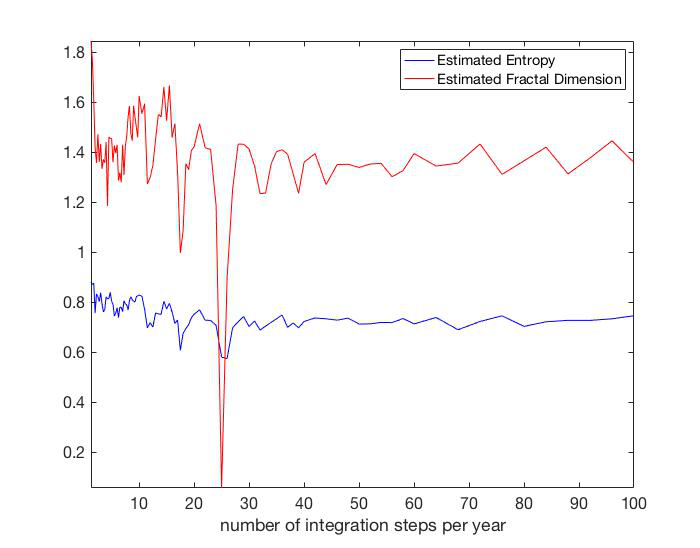}
\caption{Estimated fractal dimension (\emph{left panel}) and entropy (\emph{right panel}) for different values of $q \in [2, 100]$, on a pseudo-logarithmic scale.}
\label{fig:entropy_and_fractals}
\end{figure}

\section{Numerical study of the random model}\label{sec:raandom_num}
In this section, we will consider the random model given by Equations~\eqref{eq:model_1_prime}--\eqref{eq:model_3_prime} and we perform some numerical experiments on it. 
In Section~\ref{subsec:entropy_and_fractal_random} we investigate the chaotic nature of the 
attractor of the random dynamical system by computing the entropy and its fractal dimension. In Section~\ref{subsec:attractors_random} we show some plots of the random attractors as in Definition~\ref{def:globally}.

In order to numerically solve equations \eqref{eq:model_1_prime}--\eqref{eq:model_3_prime}
we need to follow a discretization scheme different from the one described in 
Subsection~\ref{subsec:discretization_of_the_model}
and used for the deterministic model ~\eqref{eq:model_1}--\eqref{eq:model_3}.
Indeed we need to discretize differently the dynamics of the price $(P(t))$: in particular, we use the finite-difference Euler-Maruyama method (\cite{MA55}) for the numerical solution of the SDE in Equation~\eqref{eq:model_2_prime}.

\subsection{Entropy and fractal dimension}\label{subsec:entropy_and_fractal_random}
In this section, we investigate how the metric entropy and the fractal dimension of the attractor change
as the volatility $\sigma$ varies when the parameters of the deterministic dynamical system are fixed as both in $\mathcal{H}_1$ and $\mathcal{H}_2$; see Subsection~\ref{subsec:study_one}. For the sake of clarity, we briefly describe how the entropy and the fractal dimension are computed; cfr.~\cite{Arl19}, Appendix A.3 and Appendix A.4.\\
\indent In order to compute the entropy, we start from the ``continuous'' time series of $N_r(t)$ --the reasoning can be replicated also for the price-- for $t>T_{\max}-10000$, where $T_{\max}$ is the total length of the simulation experiment i.e., $T_{\max} = 300000$. Then, we compute the empirical auto-correlation function of $N_r$ with lags $\tau \in [0,100]$ ($\tau$ is expressed in years). We get its first ``zero'' $\tau^{\star}$ as the first point where the empirical auto-correlation crosses zero, and we then check that the value of the empirical auto-correlation at $\tau=\tau^{\star}$ is smaller than an arbitrary threshold (here, $10^{-2}$). After that, we consider the sequence $Y(t)$ only at times $t$ equal to an integer multiple of $\tau^{\star}$, and compute the returns 
$$\log_{10} \Ll(\frac{Y\big((i+1)\tau^{\star}\big)}{ Y( i \tau^{\star})}\Rr)$$
and their respective signs $\eps_i \in \{-1,1\}$. 
From this (finite) sequence of signs, we compute for $K=1, \ldots, 12$ the combinatorial entropy $H_K (\tau^{\star})$ of 
$( (\varepsilon_i)_{k-K+1 \leq i \leq k} )_{k \in \N}$:
\[ 
H_K (\tau^{\star}) := - \sum_{ x \in \{-1,1\}^K } p_x \log_2(p_x) 
\qquad \mbox{where} \quad 
p_x = \P \big( (\eps_i)_{k-K+1 \leq i \leq k} = x\big)
\]
the latter probability being with respect to $k$. Finally, we plot $H_K (\tau^{\star})$ as a function of $K$ and perform a standard robust linear regression in order to estimate its slope: the slope will be eventually the estimate on the entropy of the system.\\
\indent Instead, in order to compute the fractal dimension of the attractors, we start from their three-dimensional visualization over $200000$ years, that is, for $N_r$ for instance: 
\[ \Kk = \Big\{ \big( N_r(t),N_r(t+1), N_r(t+2) \big)  \, , \, 100\,000 \leq t \leq 300\,000 \Big\}  . \]
Then, for various values of $\eps>0$, we compute the number $\widetilde{\Nn}_{\epsilon}(\Kk)$ of cubes $$\mathcal{C}_{i,j,k} = [i \eps; (i+1) \eps] \times [j \eps; (j+1) \eps] \times [k \eps; (k+1) \eps]$$ 
that contain at least one point of $\Kk$. Theoretically, the fractal dimension is the opposite of the slope of this graph at infinity. Here, since $\Kk$ is finite, $\widetilde{\Nn}_{\eps}(\Kk)$ is constant equal to $\mathrm{Card}(\Kk)$ for small $\eps$. So, we estimated the slope of the graph only for a limited set of values of $\eps$.Because of the numerous approximations made during this estimation, the precise value of the fractal dimension should not be taken into account too seriously, but its order of magnitude should be correct.\\
\indent Figure~\cite{fig:entropy} displays the results. The plots of the entropy, as expected, confirm that the randomization of the system brings more chaos and also fuzziness and, therefore, an higher dimension. However, 
the behaviour of the plot for the fractal dimension is not monotone as much as that of the entropy. This can be linked to the lack of precision in the computation of the fractal dimension, or simply to the fact that too much ``fuzziness'' leads to having many isolated points that do not contribute to a substantial increase in the fractal dimension.

\begin{figure}[t]
\begin{subfigure}{0.5\textwidth}
\center{\includegraphics[width=\textwidth]{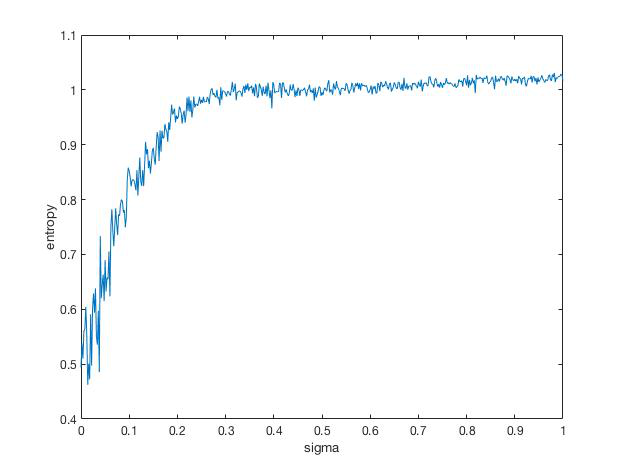}}
\caption{Entropy as function of $\sigma$ for $\Hh_1$}
\end{subfigure}
\begin{subfigure}{0.5\textwidth}
\center{\includegraphics[width=\textwidth]{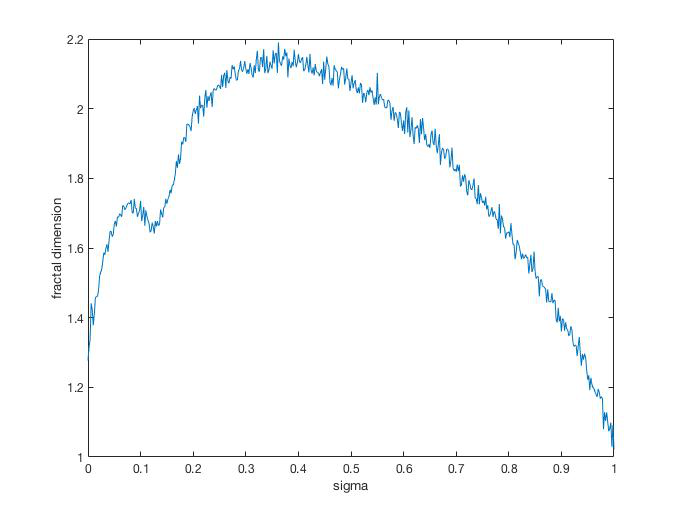}}
\caption{Fractal Dimension in terms of $\sigma$ for $\Hh_1$}
\end{subfigure}
\begin{subfigure}{0.5\textwidth}
\center{\includegraphics[width=\textwidth]{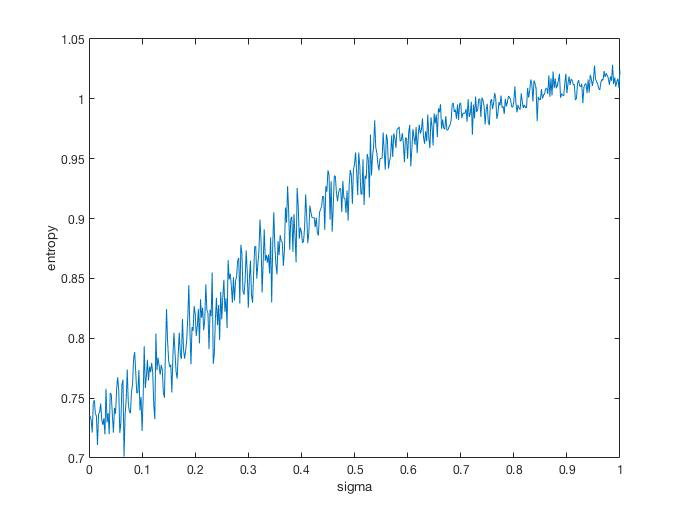}}
\caption{Entropy as function of $\sigma$ for $\Hh_2$}
\end{subfigure}
\begin{subfigure}{0.5\textwidth}
\center{\includegraphics[width=\textwidth]{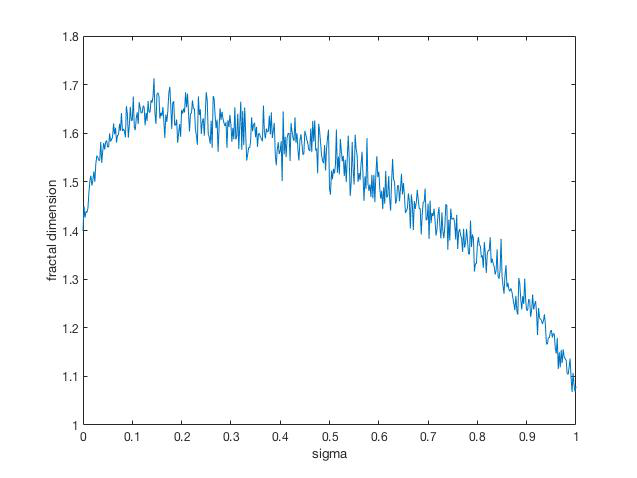}}
\caption{Fractal Dimension in terms of $\sigma$ for $\Hh_2$}
\end{subfigure}
\caption{Estimated fractal dimension and estimated entropy as a function of the parameter $\sigma$ when the parameters of the deterministic dynamical system are fixed as both in $\Hh_1$ (\emph{top panel}) and in $\Hh_2$ (\emph{bottom panel})}
\label{fig:entropy}
\end{figure}

\subsection{Plotting the random attractor}\label{subsec:attractors_random}
In this subsection, we display some plots of the globally attracting set as defined in  Definition~\ref{def:globally_new}. Notice that in Proposition~\ref{prop:existence_globally} we prove the existence of such set, which is \emph{random} and it is constructed by fixing a realization $\omega$ of the Brownian motion and then measuring the value $\phi(t,\om)(N_r^0,P^0)$ for many different initial values $(N_r^0,P^0)$. In practice, we fix a large enough time, say $t=10000$, we generate a huge amount of random initial values, say $20\,000$, and we move the system forward in time by $10000$ years, while using the same 
seed for the generation of the Brownian Motion. Because the attractor is high-dimensional, we display in Figure~\ref{fig:random_attractor} its two-dimensional projection. More precisely, we plot the points $\big(N_r(t),N_r(t+1)\big)$ for $t=10000+i/50$ with $i = 0,1,\dots,5$: we observe actually a gradual evolution of the shape. Similarly, Figure~\ref{fig:random_attractor_1} shows, instead, the same two-dimensional projection for $t=10\,000+i$ with $i=0,1,2,3$, i.e. for integer values of $t$. Now, we make the following remarks.   

\begin{remark}\label{rmk:remark_1}
Because of the definition of random attractor, should both the initial states and the path of the Brownian motion be random, we would have obtained the same plot for all the values of $t$.  
\end{remark}

\begin{remark}\label{rmk:remark_1}
Figures~\ref{fig:random_attractor} and \ref{fig:random_attractor_1} show not only the two-dimensional projection of the attractor as function of the random parameter $\omega$, but also the projection of the random measure, whose existence has been proved in Corollary~\ref{corol:existence_random_measure}. In particular, the areas with higher density are shown in yellow, while the areas of lower density are shown in blue.
\end{remark}

\noindent Finally, we observe that by fixing the seed of the Brownian motion and by letting $t$ vary, we can show every possible instance of the attractors $A(\om)$ thanks to the following lemma.
\begin{lemma}
The dynamical system $(\Om,\Ff,(\theta(t))_{t\in\R},\P)$ is ergodic.
\end{lemma}
\begin{proof}
It is a consequence of Kolmogorov’s zero-one law. Consider the two-parameter filtration $\{\Ff_s^t\}_{s\le t}$ generated by the Brownian motion, and define the $\sigma$-algebra
$$\mathcal T ^\infty = \bigcap_{t \in \R} \Ff^\infty_t.$$ 
The independence of the $\sigma$-algebras $\Ff_s^u$ and $\Ff_t^z$ for all $s < u \le t < z$ allows to apply Kolmogorov’s dichotomy and deduce that the $\sigma$-algebra $\mathcal T^\infty$ is degenerate, i.e. $\P(A) \in \{0, 1\}$ for all $A \in \mathcal T^\infty$. The conclusion follows by observing that every $\theta(t)$-invariant set is contained in $\mathcal T^\infty$.
\end{proof}

\begin{figure}[t]
\begin{subfigure}{0.324\textwidth}
\center{\includegraphics[width=\textwidth]{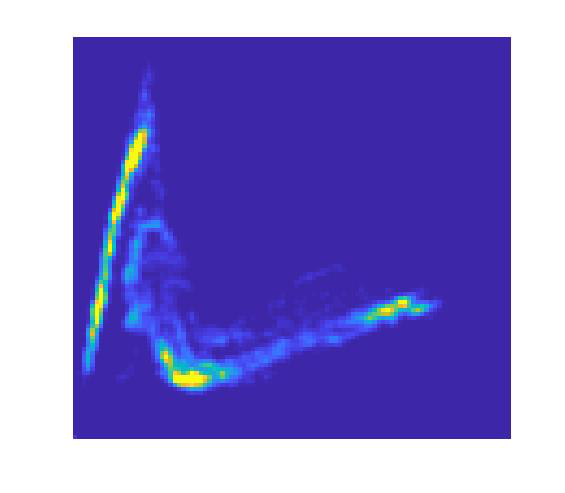}}
\caption{$t=10\,000\quad$}
\end{subfigure}
\begin{subfigure}{0.324\textwidth}
\center{\includegraphics[width=\textwidth]{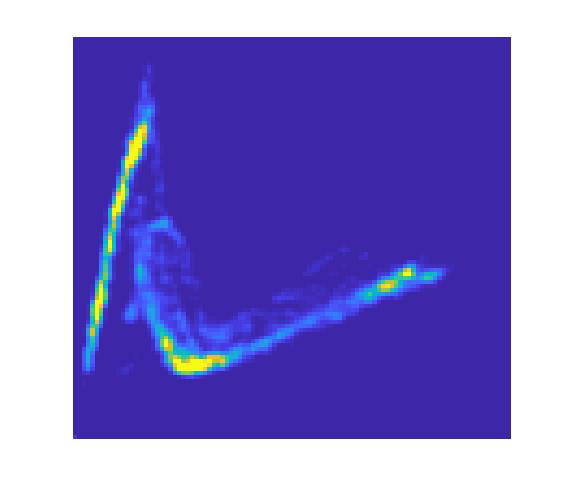}}
\caption{$t=10\,000+\frac1{50}$}
\end{subfigure}
\begin{subfigure}{0.324\textwidth}
\center{\includegraphics[width=\textwidth]{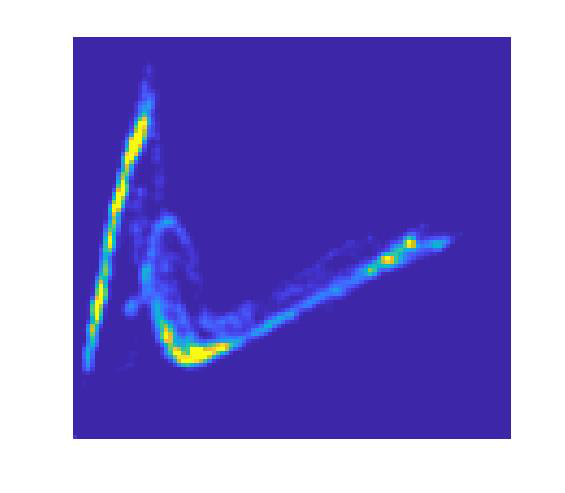}}
\caption{$t=10\,000+\frac2{50}$}
\end{subfigure}
\begin{subfigure}{0.324\textwidth}
\center{\includegraphics[width=\textwidth]{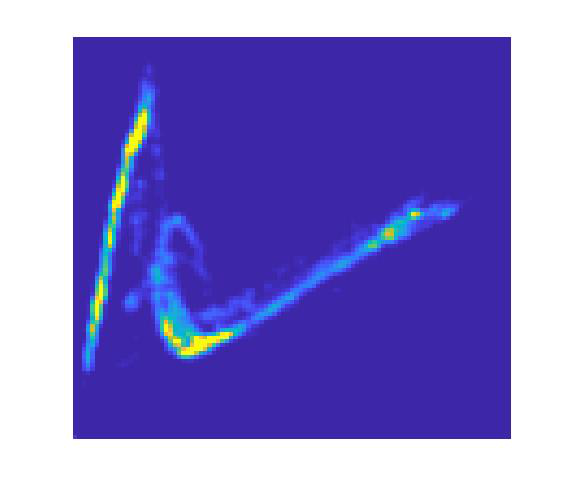}}
\caption{$t=10\,000+\frac3{50}$}
\end{subfigure}
\begin{subfigure}{0.324\textwidth}
\center{\includegraphics[width=\textwidth]{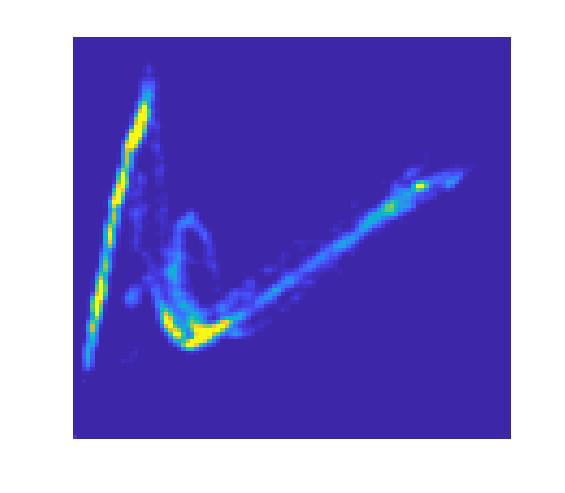}}
\caption{$t=10\,000+\frac4{50}$}
\end{subfigure}
\begin{subfigure}{0.324\textwidth}
\center{\includegraphics[width=\textwidth]{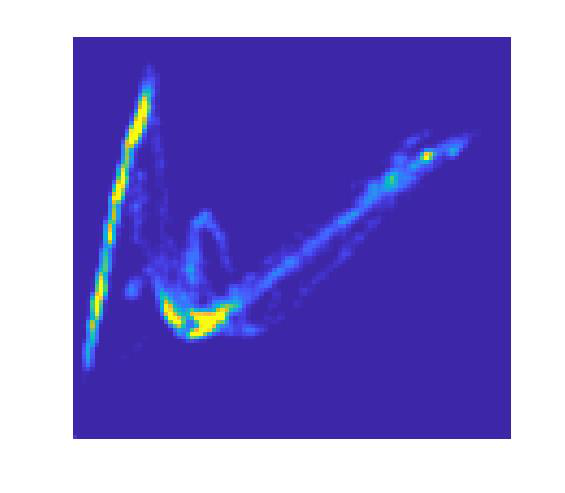}}
\caption{$t=10\,000+\frac5{50}$}
\end{subfigure}
\caption{Two dimensional projection of the random attractor.}
\label{fig:random_attractor}
\end{figure}

\begin{figure}[!ht]
\begin{subfigure}{0.5\textwidth}
\center{\includegraphics[width=\textwidth]{attr0.png}}
\caption{$t=10\,000$}
\end{subfigure}
\begin{subfigure}{0.5\textwidth}
\center{\includegraphics[width=\textwidth]{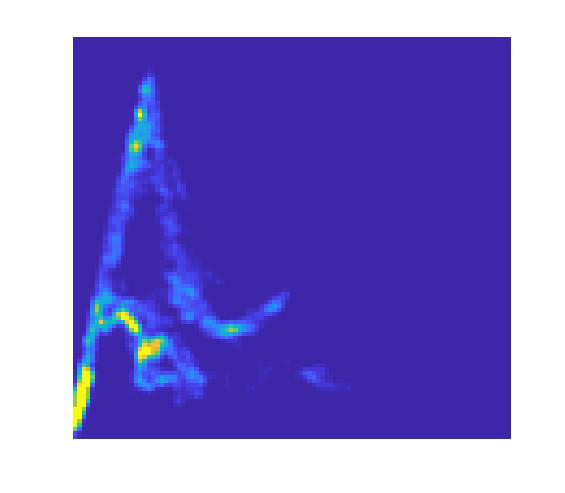}}
\caption{$t=10\,001$}
\end{subfigure}
\begin{subfigure}{0.5\textwidth}
\center{\includegraphics[width=\textwidth]{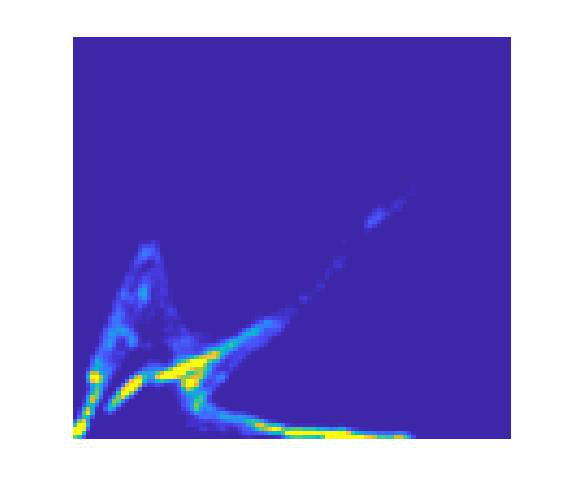}}
\caption{$t=10\,002$}
\end{subfigure}
\begin{subfigure}{0.5\textwidth}
\center{\includegraphics[width=\textwidth]{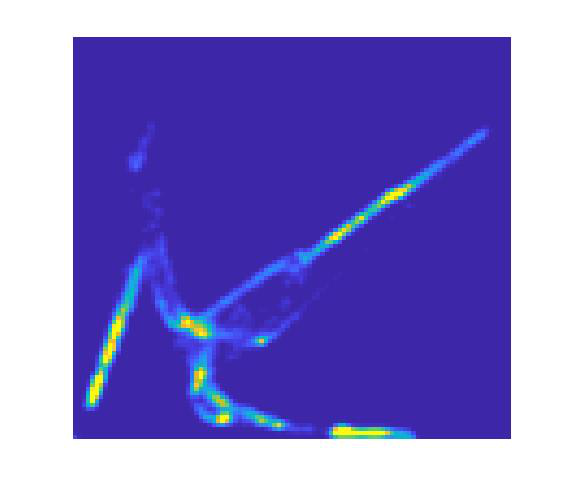}}
\caption{$t=10\,003$}
\end{subfigure}
\caption{Two dimensional projection of the random attractor (in the case of consecutive integer values of $t$).}
\label{fig:random_attractor_1}
\end{figure}

\newpage
{
\bibliographystyle{Chicago}
\bibliography{random_arXiv}
}

\newpage
\appendix
\section{Appendix A: Random Dynamical Systems, Random Attractors and Random Measures}\label{app:RDS_RA_RM}niversal approximation theorem for continuous functions between Euclidean spaces
In this appendix, we remind some fundamental notions and tools 
from the theory of random dynamical systems we have used in our analysis of the model. We refer the reader to \cite{CRFL92, FLTO19} (which we closely follow) and \cite{BH07} for further information and more details.
\noindent
\begin{definition}[Non-Autonomous Dynamical Systems (NADS)] Let $T$ be either the set of real numbers $\mathbb{R}$ or the set of integer number $\mathbb{Z}$. Let $\mathcal{X}$ be defined as in \eqref{eq:defition_of_X}. A NADS with time $T$ is a family of continuous maps $\varphi(s, t):\mathcal{X}\rightarrow\mathcal{X}$ indexed by two times $s \leq t$ with $s, t \in T$ satisfying to the following rules:  
\begin{enumerate}
    \item $\varphi(s, s) = \text{Id}_{\mathcal{X}}\,\forall\,s \in T$.
    \item $\varphi(r, t) \circ \varphi(s, r) = \varphi(s, t)\,\forall\,s \leq r \leq t$.
\end{enumerate}
\end{definition}

\noindent Notice that in the autonomous case a one-parameter family of continuous map would be sufficient to entirely determine the evolution of the system because in this case the evolution is invariant with respect to translation in time, i.e. $\varphi(s,t)x = \varphi(t-s)x$. Instead, in the non-autonomous case the time at which the initial data are prescribed is crucial; as a consequence, it is natural to expect that the ``fixed points'' of the system depend upon the second variable $t$ by letting $s \rightarrow -\infty$. We now give the notion of attractors in the non-autonomous set-up:

\begin{definition}[Pullback attractor]\label{def:pullback}
A family of objects $A(t)$ in a complete metric phase-space $(\mathcal{X},\delta )$ is a pullback attractor for the NADS $\varphi$ if it satisfies the following two conditions:
\begin{enumerate}
    \item For all $t$, $A(t)$ is a compact subset of $\Xx$ and is invariant with respect to the dynamics, namely, $\phi(s, t)A(s) = A(t)\,\forall\,s \le t$.
    \item For all bounded sets $B$, $\forall\,\eps > 0$ there exists $s_0 < 0$ such that $\forall s < s_0$ we have $\phi(s,t)(B) \subset \mathcal U_\eps (A(t))$; $U_\eps (A(t))$ denotes the neighbourhood of radius $\eps$ around the set $A(t)$, namely: 
    \begin{equation}\label{eq:neighbourhood_U_epsilon}
        \mathcal U_\eps(A(t))=\{x\in\Xx\,:\,\inf_{y\in A(t)}\delta (x,y)<\eps\}
    \end{equation}
\end{enumerate}
\end{definition}

\noindent We remark that the previous definition can be given also by means of the non-symmetric Hausdorff-like distance $\delta_h$ between sets. In this case, point \emph{2.} becomes: ``for all bounded sets $B$ it holds that $\lim_{s \rightarrow -\infty} \delta_h(\varphi(s,t)(B), A(t)) = 0$". Because of the fact that NADS and RDS are closely related, the notion of stochastic attractor will be based on the latter definition.

\subsection{Random Dynamical Systems}
\noindent We recall that a measurable dynamical system is a tuple $((\Omega, \mathcal{F}, \mathcal{P}), (\theta_t)_{t \in T})$ where $(\Omega, \mathcal{F}, \mathcal{P})$ is a probability space and $\{\theta_t : \Omega \rightarrow \Omega\}_{t \in T}$ is a family of measure preserving transformations of the probability space $(\Omega, \mathcal{F}, \mathcal{P})$ such that $(t,\omega) \rightarrow \theta_t \omega$ is measurable, $\theta_0 = \text{Id}$ and $\theta_{t+s} = \theta_t \circ \theta_s\,\forall t, s \in T$. 

\begin{definition}[RDS]
\label{def:RDS}
Let $T = \mathbb{R}, \mathbb{R}_{+}, \mathbb{Z}$ or $\mathbb{N}$. A RDS with time $T$ on a metric, complete and separable space $(\mathcal{X}, \delta )$ with Borel $\sigma$-algebra $\mathcal{B}$ over $(\theta_t)_{t \in T}$ on $(\Omega, \mathcal{F}, \mathcal{P})$ is a measurable map
\begin{equation*}
\begin{split}
    \varphi &: T \times \mathcal{X} \times \Omega \rightarrow \mathcal{X}\\
            &\quad\quad\,\,    (t, x, \omega) \rightarrow \varphi(t,\omega)x \\
\end{split}
\end{equation*}
such that $\varphi(0,\omega)=\text{id}_{\mathcal{X}}$ and 
\begin{equation}\label{eq:cocycle}
    \varphi(t+s,\omega) = \varphi(t, \theta_s\omega) \circ \varphi(s,\omega)
\end{equation}
$\forall t, s \in T$ and $\forall \omega \in \Omega$. A family of maps $\varphi(t, \omega)$ satisfying~\eqref{eq:cocycle} is called a cocycle and~\eqref{eq:cocycle} is the cocycle property. 
\end{definition}

\noindent A RDS is said to be \emph{continuous} or \emph{differentiable} if $\varphi(t, \omega):\mathcal{X}\rightarrow\mathcal{X}$ is continuous or differentiable, respectively, $\forall t \in T$ outside a $\mathcal{P}$-nullset. In addition, $\varphi(t, \omega)$ is automatically invertible if $T = \mathbb{R}$ or $\mathbb{Z}$; indeed, in this case, we have $\varphi(t, \omega)^{-1} = \varphi(-t, \theta_t\omega)$ for $t \in T$.\\
\noindent The notion of skew product collects all the $\omega's$ in order to define a (measurable) dynamical system on the product space $(\Omega \times \mathcal{X}, \mathcal{F} \otimes \mathcal{B})$:
\begin{definition}[Skew product]
\label{def:skew_product}
The measurable map
\begin{equation*}
\begin{split}
    \Theta &: T \times \Omega \times \mathcal{X} \rightarrow \Omega \times \mathcal{X}\\
           &\quad\quad\,\,(t, \omega, x) \rightarrow (\theta(t)\omega, \varphi(t, \omega)x),
\end{split}
\end{equation*}
is called the skew product flow of the dynamical system $((\Omega, \mathcal{F}, \mathcal{P}), (\theta_t)_{t \in T})$ and of the co-cycle $\varphi$. 
\end{definition}
\noindent In particular, the family of mapping $\Theta_t = \Theta(t,\,\cdot\,,\,\cdot\,)$ with $t \in T$ is the measurable dynamical system on $(\Omega \times \mathcal{X}, \mathcal{F} \otimes \mathcal{B})$ we were referring to. In addition, it holds that $\Theta_{0} = \text{Id}_{\Omega \times \mathcal{X}}$ and $\Theta_{t+s}(\omega, x) = \Theta_t \circ \Theta_s(\omega, x)\,\forall\,t, s \in T, \omega \in \Omega\,\text{and}\,x\in \mathcal{X}$.

\subsubsection{Attraction and absorption}
We here define the notions of attraction and absorption. 
\begin{definition}
A random set $A$ is said to attract another random set $B$ if $\mathcal{P}$-almost surely
\begin{equation*}
    \lim_{t \rightarrow \infty} \delta_h(\varphi(t, \theta_{-t}\omega)B(\theta_{-t}\omega), A(\omega)) = 0.
\end{equation*}
\end{definition}
\noindent Moreover, we have that a random set $K(\omega)$ is said to be (\emph{strictly}) $\varphi$-forward invariant if 
\begin{equation*}
   \varphi(t, \omega)K(\omega) \subset K(\theta_t\omega)\quad(\varphi(t, \omega)K(\omega) = K(\theta_t\omega))\quad\forall t > 0.
\end{equation*}
In addition, the following definitions hold.
\begin{definition}[Globally attracting set]
\label{def:globally}
Suppose $\varphi$ is a RDS such that there exists a random compact set $A(\omega)$ which satisfies the following conditions:
\begin{enumerate}
    \item $\varphi(t, \omega)A(\omega) = A(\theta_t\omega)\,\forall\,t>0$.
    \item $A$ attracts every bounded deterministic set $B \subset \mathcal{X}$.
\end{enumerate}
Tnen, $A$ is said to be a universally or globally attracting set for $\varphi$.
\end{definition}

\begin{definition}[Absorption time]
If $K$ and $B$ are random sets such that for $\mathcal{P}$-almost all $\omega$ there exists a time $t_B(\omega)$ such that for all $t \geq t_{B}(\omega)$ we have
\begin{equation*}
    \varphi(t, \theta_{-t}\omega)B(\theta_{-t}\omega) \subset K(\omega),
\end{equation*}
then $K$ is said to absorb $B$ and $t_{B}(\omega)$ is called the absorption time.
\end{definition}

\begin{definition}[$\Omega$-limit set]
Given a random set $K$, the set 
\begin{equation*}
    \Omega(K, \omega) = \Omega_{K}(\omega) = \bigcap_{T \geq 0}\overline{\bigcap_{t \geq T}\varphi(t, \theta_{-t}\omega)K(\theta_{-t}\omega)}
\end{equation*}
is said to be the $\Omega$-limit set of $K$. By definition, $\Omega_{K}(\omega)$ is closed.
\end{definition}
In particular, it is possible to identify $\Omega_{K}(\omega)$ with
\begin{equation*}
    \Omega_K(\omega) = \{y \in \mathcal{X}\,:\,\exists\,t_n\rightarrow\infty, x_n \in K(\theta_{-t_n}\omega)\,:\,\varphi(t_n, \theta_{-t_n}\omega)x_n \rightarrow y\}.
\end{equation*}
With this identification, the $\theta$-shift of an $\Omega$-limit set is given by:
\begin{multline*}
\Omega_K(\cdot)\circ\theta_t =\Omega(K,\theta_t\om)= \{y \in \Xx: \exists t_n\to \infty,\\\exists x_n \in K(\theta_{-t_n+t}\om)\ \text{such that} \ \phi(t_n, \theta_{-t_n+t}\om)x_n \to y\}.
\end{multline*}
Also, the following theorem holds:
\begin{atheorem}[cfr.~\cite{CRFL92}, Theorem 3.11]
\label{th:global_attractor}
Suppose $\varphi$ is an RDS on the Polish space $\mathcal{X}$ and suppose that there exists a compact set $K(\omega)$ absorbing every bounded non-random set $B \subset \mathcal{X}$. Then the set
\begin{equation*}
    A(\omega) = \overline{\bigcup_{B \subset \mathcal{X}}\Omega_{B}(\omega)}
\end{equation*}
is a global attractor for $\varphi$.
\end{atheorem}

\subsubsection{Invariant measures on random sets}
\label{subsec:invariant}
First, we introduce the concept of random (probability) measure on $\mathcal{X}$. 
\begin{definition}[Random (probability) measure]
A map $\mu: \mathcal{B} \times \Omega \rightarrow [0,1]$, $(B, \omega) \rightarrow \mu_{\omega}(B)$ satisfying the following two conditions
\begin{enumerate}
    \item for every $B \in \mathcal{B}$, $\omega \rightarrow \mu_{\omega}(B)$ is measurable, 
    \item for $\mathcal{P}$-almost every $\omega \in \Omega$, $B \rightarrow \mu_{\omega}(B)$ is a Borel probability measure,
\end{enumerate}
is said to be a random (probability) measure on $\mathcal{X}$. 
\end{definition}
\noindent Hereafter, we will denote by $\mathcal{P}_{\Omega}(\mathcal{X})$ the set of random measures on $\mathcal{X}$ and by $\mathcal{P}(\mathcal{X})$ the set of canonical probability measures. 

\begin{definition}[Invariant measure]
A random (probability) measure $\mu = (\mu_{\omega})_{\omega \in \Omega}$ is said to be \emph{invariant} for the RDS $\varphi$ if $\forall t \in T$ and for $\mathcal{P}$-almost every $\omega \in \Omega$ we have
\begin{equation*}
    \varphi(t,\omega)\mu_{\omega} = \mu_{\theta(t)\omega}.
\end{equation*}
\end{definition}
In this work, we will aim at finding an invariant measure via an averaging procedure applied to the initial distribution of the data. To this end, let $\lambda_{\mu}$ be the measure of which $\mu$ is the \emph{factorization}, i.e., the measure defined $\forall A \in \mathcal{F} \otimes \mathcal{B}$ as
\begin{equation}\label{eq:factorization_mu}
    \lambda_{\mu}(A) = \int_{\Omega}\left(\int_{\mathcal{X}} I_{A}(\omega, x)\mu_{\omega}(\,dx)\right)\mathcal{P}(\,dx).
\end{equation}
The concept of invariance for a random (probability) measure $\mu$ is related to the invariance of the measure $\lambda_{\mu}$ with respect to the skew product; the following proposition holds:
\begin{proposition}[cfr.~\cite{FLTO19}, Proposition 72]
The random (probability) measure $\mu = (\mu_{\omega})_{\omega \in \Omega}$ is invariant for the RDS $\varphi$ if and only if the measure $\lambda_{\mu}$ on $(\Omega \times \mathcal{X}, \mathcal{F} \otimes \mathcal{B})$ is invariant for the skew product associated to $\varphi$.
\end{proposition}
 
\noindent Hereafter, we will denote by $\mathcal{P}_{\mathcal{P}}(\Omega \times \mathcal{X})$ the set of probability measures on $(\Omega \times \mathcal{X}, \mathcal{F} \otimes \mathcal{B})$ of the form \eqref{eq:factorization_mu} for some random (probability) measure $\mu$. In particular, $\Theta_t$ maps $\mathcal{P}_{\mathcal{P}}(\Omega \times \mathcal{X})$ into itself. Also, it can be shown that all the probability measures on $(\Omega \times \mathcal{X})$ with marginal $\mathcal{P}$ on $\Omega$ have a unique random (probability) measure satisfying Equation~\eqref{eq:factorization_mu}. This leads to a one-to-one correspondence between $\mathcal{P}_{\mathcal{P}}(\Omega \times \mathcal{X})$ and $\mathcal{P}_{\Omega}(\mathcal{X})$. In addition, we denote by $\mathcal{I}_{\mathcal{P}}(\varphi) \subset \mathcal{P}_{\mathcal{P}}(\Omega \times \mathcal{X})$ the set of measures $\lambda_{\mu}$ for which the associated random measure $\mu$ is invariant. We give now the following definition.

\begin{definition}[cfr.~\cite{FLTO19}, Page 51]
We define $L_{\mathcal{P}}^{1}(\omega, \mathcal{C}_{b}(\mathcal{X}))$ as the space of those functions $f:\Omega\rightarrow\mathcal{C}_b(\mathcal{X})$ such that the map $(\omega, x) \rightarrow f(\omega)(x) = f(\omega, x)$ is measurable and the integral
\begin{equation*}
    \|f\|_{1,\infty}:=\int_{\Omega} \sup_{x \in \mathcal{X}} |f(\omega, x)| d\mathcal{P}(\omega)
\end{equation*}
is finite.  
\end{definition}
\noindent In particular, we identify two functions $f$ and $g$ if $\mathcal{P}(f(\,\cdot\,,\omega) \neq g(\,\cdot\,,\omega))=0$; the equivalence class of $f$ will be identified with $f$.\\
\noindent The space $\mathcal{P}_{\mathcal{P}}(\Omega \times \mathcal{X})$ is endowed with the topology of the weak convergence, which is the smallest topology on $\mathcal{P}_{\Omega}(\mathcal{X})$ such that the maps
\begin{equation*}
    \mu \rightarrow \mu(f) = \int_{\Omega}\int_{\mathcal{X}} f(\omega, x)\,d\mu_{\omega}(x)\,d\mathcal{P}(\omega) = \int_{\Omega \times \mathcal{X}} f(\omega, x)\,d\lambda_{\mu}(\omega, c)
\end{equation*}
are continuous for each $f \in L_{\mathcal{P}}^{1}(\omega, \mathcal{C}_{b}(\mathcal{X}))$. At this point, we consider the action of the skew product (see Definition~\ref{def:skew_product}) on functions $f \in L_{\mathcal{P}}^{1}(\omega, \mathcal{C}_{b}(\mathcal{X}))$ given by $\Theta_t f = f \circ \Theta_t$; in particular, such a product belongs to $L_{\mathcal{P}}^{1}(\omega, \mathcal{C}_{b}(\mathcal{X}))$ too.

\begin{proposition}[cfr.~\cite{FLTO19}, Proposition 73]
If $\varphi$ is a continuous RDS on a Polish space $\mathcal{X}$, the map $\mu \rightarrow \Theta_t \mu$ on $\mathcal{P}_{\mathcal{P}}(\Omega \times \mathcal{X})$ is affine and continuous. Moreover, the set $\mathcal{I}_{\mathcal{P}}(\varphi)$ is convex and closed.
\end{proposition}

\noindent We have now all the theoretical instruments to introduce the averaging process mentioned before, which enables us to state the existence of measures in $\mathcal{I}_{\mathcal{P}}(\varphi)$.    

\begin{proposition}[cfr.~\cite{FLTO19}, Proposition 74]
\label{prop:averaging}
Let $\varphi$ be a continuous RDS on a Polish space $\mathcal{X}$ with continuous time $T$. Let $\nu$ be in $\mathcal{P}_{\mathcal{P}}(\Omega \times \mathcal{X})$. For each $t \in T$, $t > 0$, let $\mu_t$ be the measure defined as 
\begin{equation}\label{eq:averaging_measure}
    \mu_t(A) := \frac{1}{t} \int_{0}^{t} (\Theta_s \nu)(A)\,ds,\quad \forall A \in \mathcal{F} \otimes \mathcal{B}.
\end{equation}
Then, every limit point of $(\mu_t)_{t}$ for $t\rightarrow\infty$, in the topology of the weak convergence, is in $\mathcal{I}_{\mathcal{P}}(\varphi)$. 
\end{proposition}

\noindent Finally, the existence of limit points for the sequence \eqref{eq:averaging_measure}, and then of random invariant measures for $\varphi$, can be established through an analogous of Prohorov theorem for random measures.
\begin{definition}
A set of measures $\Gamma\subset \Pp_\P (\Om \times \Xx)$ is said to be tight if for every $\eps > 0$ there exists a compact set $C_\eps \subset \Xx$ such that, for each $\lam \in \Gamma$ it holds that $\lam(\Om\times C_\eps) \ge 1-\eps$. In other words, we must have $\int_\Omega \mu_\om(C_\eps) \P(\diff \om)\ge1-\eps$, where $\mu_\om$ is the factorization of $\lambda_\mu.$
\end{definition}

\begin{theorem}[Prohorov theorem for Random Measures, cfr.~\cite{CR02}, Theorem 4.4]\label{th:prohorov_random}
If $\Gamma\subset \Pp_\P (\Om \times \Xx)$ is tight, then every sequence $(\mu_n)_{n\in\N} \subset \Gamma$ admits a convergent sub-sequence.
\end{theorem}

\section{Appendix B: Measuring distances between attractors: the logistic map case}\label{app:logistic_attractors}
We show the reliability of the methodology employed in Subsection \ref{subsec:distances} to quantify the distance between two attractors by applying it to the logistic map, i.e. the map $T(t) = r t (1-t)$, with $r \in [3.5, 4]$ and $\# r = 3000$. For each parameter $r$, we let the system evolves for $50000$ steps forward, and we record the position of the system in $3000$ histograms with $500$ bins. We then apply the four presented distances to all of the histograms on the first argument, and to the penultimate histogram on the right argument. The choice of the penultimate histogram is due to the fact that the dynamics  when the logistic parameter $r=4$ is somewhat simpler, being conjugated to the tent map. Thus, we picked the histogram of the more ``chaotic'' logistic map as a second argument. All the plots have a rather uniform shape, with some spikes corresponding to the stable windows in the bifurcation diagram of the logistic map, and the distance gradually decreasing as the parameter approaches the value of $4$.

\begin{figure}[!h]
\includegraphics[width=0.50\textwidth]{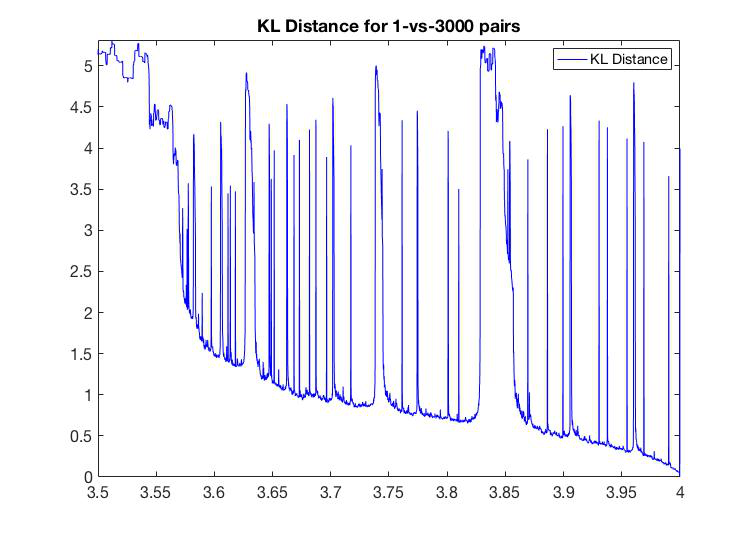}
\includegraphics[width=0.50\textwidth]{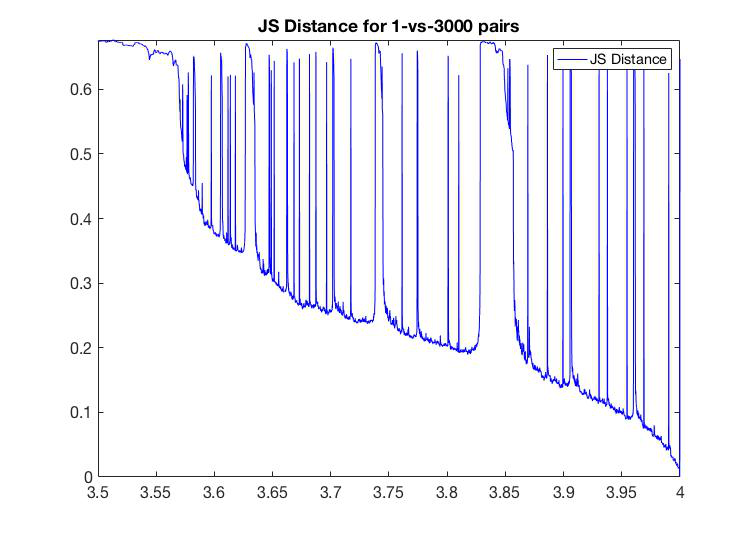}
\includegraphics[width=0.50\textwidth]{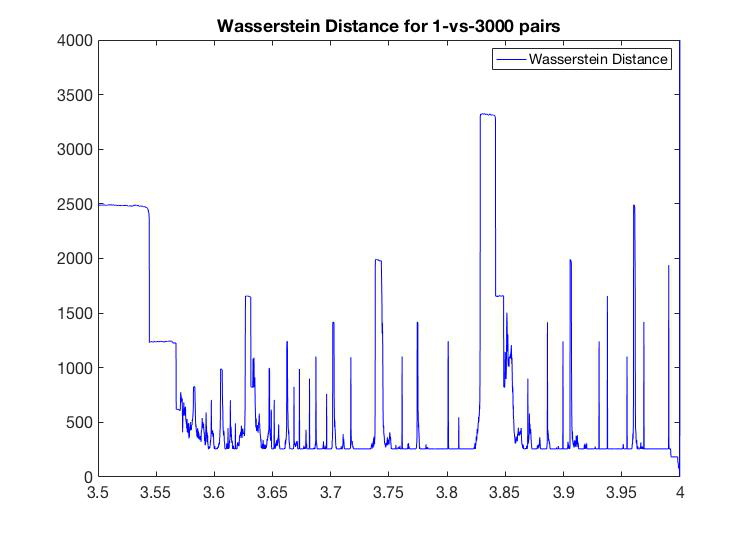}
\includegraphics[width=0.50\textwidth]{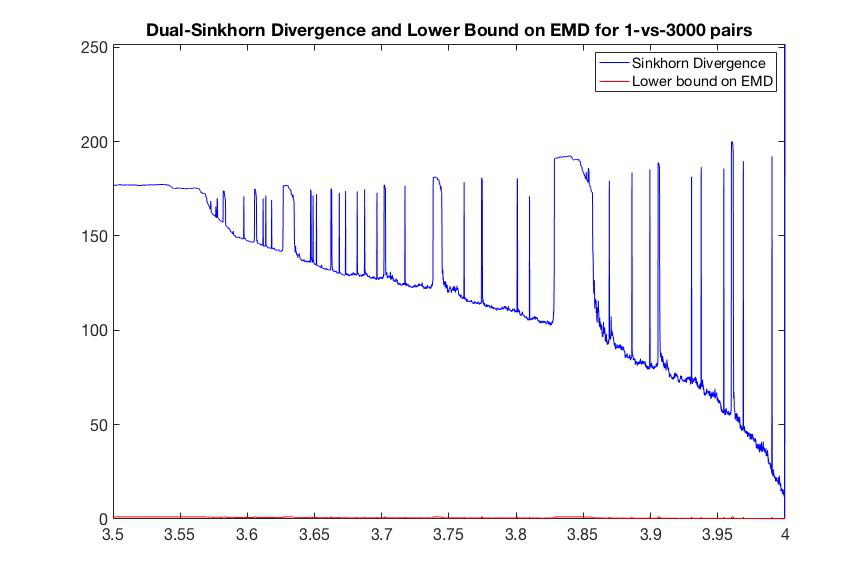}
\caption{\emph{From top to bottom, left panel:} distances between the attractors as a function of $r \in [3.5, 4]$ for the logistic map. The distances are, in order: the Kullback-Leibler divergence, the Wasserstein, the Jensen-Shannon distance and the Sinkhorn distance.}
\label{fig:distances}
\end{figure}


\newpage

\end{document}